\documentclass[10pt,reqno]{article}
\usepackage[a4paper,inner=35mm,outer=35mm,top=35mm,bottom=35mm]{geometry}
\usepackage{fullpage}
\usepackage{times}
\usepackage{authblk}
\usepackage{amsfonts,amssymb,amsmath,amsthm,mathrsfs}
\usepackage{enumerate}
\usepackage[inline]{enumitem}
\usepackage{graphicx,xcolor}
\usepackage[numbers]{natbib}

\usepackage{titlesec}

\titleformat*{\section}{\large\bfseries}
\titleformat*{\subsection}{\normalsize\bfseries}

\newtheorem{theorem}{Theorem}[section]
\newtheorem{assumption}[theorem]{Assumption}
\newtheorem{lemma}[theorem]{Lemma}

\newtheorem{corollary}[theorem]{Corollary}
\newtheorem{proposition}[theorem]{Proposition}
\newtheorem{definition}[theorem]{Definition}
\newtheorem{example}[theorem]{Example}
\newtheorem{remark}[theorem]{Remark}
\newtheorem{notation}[theorem]{Notation}

\numberwithin{equation}{section}
\newcommand{\cA}{\mathcal{A}}
\newcommand{\cB}{\mathcal{B}}
\newcommand{\cC}{\mathcal{C}}

\newcommand{\cE}{\mathcal{E}}
\newcommand{\cF}{\mathcal{F}}
\newcommand{\cG}{\mathcal{G}}
\newcommand{\cH}{\mathcal{H}}

\newcommand{\cK}{\mathcal{K}}
\newcommand{\cL}{\mathcal{L}}
\newcommand{\cM}{\mathcal{M}}

\newcommand{\cO}{\mathcal{O}}

\newcommand{\cV}{\mathcal{V}}

\newcommand{\cX}{\mathcal{X}}

\def \E{\mathbb{E}}
\def \F{\mathbb{F}}

\def \N{\mathbb{N}}
\def \P{\mathbb{P}}
\def \Q{\mathbb{Q}}
\def \R{\mathbb{R}}

\def \G{\mathbb{G}}

\def \W{\mathbb{W}}

\def \fm{\mathfrak{m}}

\def \fA{\mathfrak{A}}

\def \fM{\mathfrak{M}}

\DeclareMathOperator{\sign}{sign}

\newcommand{\prob}{\mathcal{M}_{1}} 						
\newcommand{\subp}{\mathcal{M}_{\le 1}} 					
\newcommand{\law}{\cL} 										
\newcommand{\smL}{\mathfrak{L}} 							
\newcommand{\tauex}{\tau_{\mathrm{x}}} 						
\newcommand{\tauey}{\tau_{\mathrm{y}}} 						
\newcommand{\Xtauex}{X_{\tauex}} 							
\newcommand{\env}{\mathfrak{P}}								
\newcommand{\trust}{Y}										
\newcommand{\fxspace}{\cE}									
\newcommand{\sexp}{\mathscr{E}}								

\makeatletter
\makeatletter
\renewcommand*{\@fnsymbol}[1]{\ensuremath{\ifcase#1\or *\or **\or 
\mathsection\or \mathparagraph\or \|\or **\or \dagger\dagger
\or \ddagger\ddagger \else\@ctrerr\fi}}
\makeatother

\begin{document}

\title{Mean field games with absorption and common noise \\ with a model of bank run}
\author{Matteo Burzoni and Luciano Campi}
\affil{Universit\`a degli Studi di Milano}
\date{\today}
\maketitle
\abstract{We consider a mean field game describing the limit of a stochastic differential game of $N$-players whose state dynamics are subject to idiosyncratic and common noise and that can be absorbed when they hit a prescribed region of the state space. We provide a general result for the existence of weak mean field equilibria which, due to the absorption and the common noise, are given by random flow of sub-probabilities. We first use a fixed point argument to find solutions to the mean field problem in a reduced setting resulting from a discretization procedure and then we prove convergence of such equilibria to the desired solution. We exploit these ideas also to construct $\varepsilon$-Nash equilibria for the $N$-player game. Since the approximation is two-fold, one given by the mean field limit and one given by the discretization, some suitable convergence results are needed. We also introduce and discuss a novel model of bank run that can be studied within this framework.}


\section{Introduction}
Mean field games are natural limits of symmetric stochastic differential $N$-player games when the number of players goes to infinity. The interaction between the players is in a weak sense, i.e. each player interacts with the rest of the population through its empirical measure. Mean field games have been introduced simultaneously by Lasry and Lions \cite{lasry-lions2006a,lasry-lions2006b,lasry-lions2007} and Huang et al. \cite{huang2006} in order to overcome the curse of dimensionality. Indeed, solving such games when $N$ is fixed and very large can be a cumbersome task, while passing to the limit with the number of players simplifies the analysis leading to approximate Nash equilibria for the original $N$-player games with vanishing error for $N \to \infty$. We refer to the two-volume book by Carmona and Delarue \cite{bookMFG} for a thorough and detailed treatment of the probabilistic approach to mean field games, while Cardaliaguet \cite{cardaliaguet} provides a complete introduction of the PDE-based techniques. For a recent survey of the many applications of such a class of games, especially in economics and finance, we refer to \cite{carmona2020applications}.

In this paper we consider a mean field game with \emph{absorption} and \emph{common noise}, arising as the limit of a stochastic differential game with weakly interacting players who can be absorbed, hence disappearing from the game, when their private states hit the boundary of a given domain. In the case without common noise, those games have been introduced in \cite{CF18} and further extended and analysed in \cite{CGL19}. In the $N$-player game, the private states of each player interact through the empirical distribution of the remaining players as well as through the number of absorbed players in a smooth fashion as in \cite{HS19}. Players' goal is maximising some objective functional up to the absorption time or a given maturity, whichever comes first. We are interested in showing existence of a solution for the limit mean field game and to study how such a solution can be used to provide approximate Nash equilibria for the $N$-player game for $N$ sufficiently large. Related to this problem is the study of the absorption phenomenon for McKean-Vlasov dynamics as in \cite{giesecke2013default,giesecke2015large,hambly2017stochastic,hambly2019mckean,HS19,ledger2020uniqueness,nadtochiy2019particle,nadtochiy2020mean}. A detailed description of the way these papers relate to mean field games with absorption can be found in \cite{CF18,CGL19}.

The structure of the mean field game we are interested in here is motivated by a novel bank-run model that we introduce in this paper. Other applications are also possible, e.g. to system of interacting defaultable firms or banks maximizing their profits or minimizing their reserves, respectively. Different mean field games approaches to the problem of bank run can be found in \cite{Carmona_bank,Nutz_stopping}. In the first paper the reserves of the bank are not affected directly, and in a dynamic way, by the run of investors. The key factor is whether the bank, at a certain given time, would be able to repay the investors who ran (see also \cite[Sect. 1.2.2, Vol. I]{bookMFG}. The second paper uses a totally different approach and the bank run problem is phrased as an optimal stopping problem. In some cases, an explicit solution is provided. In this work we propose to study a continuous-time model of bank run using a mean field games approach with absorption in the spirit of \cite{CF18,CGL19}. In a nutshell, the model consists in finitely many depositors of a given bank, whose assets' value is modelled as non-homogeneous diffusion driven by a Brownian motion and affected by the average number of leavers. Each agent chooses her own level of trust in the bank, that is also affected by an idiosyncratic noise, the bank assets' value and the average number of leavers. The goal of every agent is to maximise the largest portion of her deposit (including interests) that can be paid by the bank upon run. We will see later that such a game can be naturally framed as a $N$-player stochastic differential game with weak interaction, absorption and common noise, so that our theoretical results will apply to this model too.\\

\noindent The main contributions of this paper can be shortly summarized as follows:\begin{itemize}
\item \emph{Existence of a mean field game solution}: under some boundedness and regularity assumptions of the coefficients and the set of non-absorbing states, we prove existence of a \emph{weak} mean field solution, that is a tuple $(\xi,W,W^0,\alpha, \mu)$, where $\xi$ is the initial state, $W,W^0$ are the idiosyncratic and the common noise, respectively, $\alpha$ an admissible control and $\mu$ is an equilibrium (random) flow of sub-probability measures: given $\mu$ as input, $\alpha$ is an optimal control and each $\mu_t $ coincides a.s. with the marginal of the optimally controlled state $X^\alpha$, restricted on not being absorbed up to time $t$. The existence is proved adopting the approach of \cite{CDL16} which is based on a two-steps procedure: first, an equilibrium is derived in a reduced setting where the input $\mu$ can only take finitely many values depending on a discretized version of the common noise; second, by a suitable limiting argument, it is proved that the sequence of reduced equilibria converges to the desired solution.
\item \emph{Approximate Nash equilibria in the $N$-player game}: after establishing existence of a solution to the mean field game, we show how such a solution can be implemented in the $N$-player game, hence providing an approximate Nash equilibrium. This is done using a different approach than what it is usually done in the literature (see, e.g., \cite[Sect. 6.1.2, Vol. II]{bookMFG}). It consists in the following two steps: first, by relying on the discretization procedure used in the existence part, we show the approximation result for an $N$-player game where the interaction with the empirical distribution is reduced to a random sub-probability depending on a discretized version of the common noise; second, we pass to the limit in the discretization step and obtain the result for the original $N$-player game. The second step is pretty delicate as it requires some suitable uniform convergence results to deal with the mean field game approximation and the discretization simultaneously.
\end{itemize}

We can therefore summarize the main novelties of the paper as: the introduction of common noise in the framework of mean field games with absorption; the use of a new approach to derive approximate Nash equilibria by means of an auxiliary $N$-player game; the introduction of a novel model of bank run for which our results are applicable.
The paper is organized as follows: in Section \ref{sec:motivation} we describe the bank-run model motivating our study, in Section \ref{sec:MFG} we set the general framework together with the relevant assumptions, in Section \ref{sec:weakeq} we prove the existence result for mean field game solutions and, finally, in Section \ref{sec:approx-nash} we show how such solutions can be used to construct approximate Nash equilibria in the $N$-player game for $N$ large enough. The Appendix provides, for reader's convenience, some well-known technical results on BSDEs with (bounded) random horizon.

\begin{notation} 
	We denote by $\subp(\cX)$ the space of sub-probabilities on a metric space $(\cX,d)$, i.e., non-negative measures $\mu$ with $\mu(\cX)\le 1$. Unless otherwise specified, the space $\subp(\cX)$ will be equipped with the topology induced by the weak convergence of measures. For $\mu\in\subp(\cX)$ and $f\in L^1(\mu)$, we use the notation $\langle f,\mu\rangle:=\int f d\mu$ for denoting the integral of $f$ with respect to the measure $\mu$. The class of finite measures, resp. probability measures, on  $(\cX,d)$ is denoted by $\cM(\cX)$, resp. $\prob(\cX)$.
\end{notation}
\section{A motivating example: a model of bank run}\label{sec:motivation}
Consider a group of $N\in\N$ agents, each of which has deposited an initial capital $D^i_0$ in a given bank. The bank offers an interest rate on the deposits equal to $r$ and the amount $D^i_T:=e^{rT}D^i_0$ is promised to be paid back at time $T>0$. 
Each depositor has the right to withdraw early her capital (\emph{run}) and to collect the cumulative interests. There is a risk that the bank is not able to fulfill its obligations and, thus, not all the depositors can be paid back. 
We interpret runs as the loss of confidence for the depositor that the bank will be able to pay back the capital at time $T$. To this aim we model the level of trust of agent $i$ as a stochastic process $\trust^i=(\trust^i_t)_{t\in[0,T]}$ with initial value $Y^i_0>0$. If $\trust^i$ hits $0$, agent $i$ does not trust the bank anymore, she withdraws the deposit and the cumulative interests (if the bank has enough capital) and runs.
We allow the process $\trust^i$ to depend on the fractions of agents who already left the game; this allows us to capture the self-exciting aspect of runs (see the discussion in Section \ref{sec:discussion} below).
The market value of the assets of the bank is represented by a stochastic process $S=(S_t)_{t\in[0,T]}$. For player $i$ the game continues until time $T$ or until either $Y^i$ or $S$ hit $0$, which means that the player does not trust the bank anymore or that the bank is bankrupt and no further payments can be disbursed.
We thus consider the following system of stochastic differential equations (SDEs):
\begin{eqnarray}\label{eq:model}
dS_t &=& b^0(t,S_t)dt + \sigma^0(t,S_t) dW^0_t+e^{rt}dL^N_t \nonumber \\
d\trust^i_t&=&b^i(t,\alpha_t,\trust^i_t, S_t,L^N_t)dt+\sigma^i dW^i_t
\end{eqnarray}
for $i=1,\ldots,N$, with
\[
L^N_t=\frac{1}{N}\sum_{i=1}^N 1_{[0,\tau^i)}(t),\qquad\tau^i:=\inf\{t\in[0,T]\ :\  \trust^i_t\notin [0,\infty)\text{ or } S_t\notin [0,\infty)\}
\]
and $(S_0, \trust^1_0,\ldots,\trust^N_0)\in\R^{N+1}_+$, where $\R_+ := (0,+\infty)$. The objective function that each agent aims to maximize is
\[
J^i(\alpha)=\E\big[(e^{r\tau^i}D^i_0)\wedge S_{\tau^i}\big],
\]
namely, the largest portion of her deposit (comprehensive of interests) that can be paid by the bank upon run. This amount may be zero if bankruptcy occurs before the run.

We will study models of mean field type which can be seen as the limiting problem of the one described above. Aiming at obtaining a tractable model we introduce two approximations. First, together with a terminal payoff $G$ we also consider a running cost $f$ which can be interpreted as the cost of the effort to change opinion on the bank (e.g.\ gathering information). Second, we take a smooth approximation of the term $dL^N$ as in \cite{HS19}.
\begin{definition}\label{def:kernel}
	A smoothing kernel is a $C^\infty$ function $k:\R_+\to\R_+$ with compact support on $[0,\varepsilon]$ for some $\varepsilon>0$ and with $\int_0^\varepsilon k(u)du=1$.
\end{definition} An example of such a kernel is given by the bump function 
\[
k(u)=\frac{1}{c}\begin{cases}e^{-\frac{1}{\varepsilon-u}}&0\le u<\varepsilon\\
	0& u\ge\varepsilon,
\end{cases} 
\]
where $c$ is a suitable normalization constant. We use such kernels to define:
\[
\smL^N_t:=\int_0^tk(t-s)L^N_s ds.
\] 
When the level of trust $\trust^i$ reaches zero, for some player $i$ at time $t$, the term $dL^N_t$ would cause an immediate jump of size $1/N$ in the dynamics of $S$. By using $d\smL^N_t$ instead, the decrease of value in $S$ is not instantaneous but rather takes place over a time interval of length $\varepsilon$. This is not only useful for the tractability of the model but it also represents a more realistic situation where the bank needs a certain amount of time (typically small) in order to repay a running depositor.
\subsection{Discussion of the model}\label{sec:discussion}
Since we want to study mean field models able to approximate the above system, we think of a bank with a very large number of similar clients. In particular, the coefficients are the same for every $i$, i.e.\ $b^i\equiv b$ and $\sigma^i\equiv \sigma$, as well as the initial deposit which is normalized to $D_0=1/N$. 
\begin{itemize} 
	\item \emph{The bank's value}. The dynamics of the asset $S$ reveals that the event of a running agent will cause a negative jump in its value, in the amount of the initial deposit plus the due cumulative interests. Note that if $S_t\ge L^N_te^{rt}$, every agent still in the game can be paid in full so that the bank is solvent. The bank is insolvent (and goes bankrupt) when $S$ hits $0$. If $S_t< L^N_te^{rt}$ the bank could not survive if each of the remained agents decided to leave. In this case we can speak of technical default but if the agents do not withdraw their deposits, the bank continues to operate. In the determination of the solvency situation of the bank we can also use a deterministic function $L$ which represents its liquidation value (see \cite{Carmona_bank}).
	\item \emph{The trust dynamics}. In order to present the main ideas it is sufficient to consider a very simple instance of the above model, such as,  $d\trust_t=-\alpha_t dt$ with initial value $\trust_0=1$. In this model, a choice of $\alpha\le 1/T$ leads to the deterministic exit time $T$. This means that the depositor has the possibility of remaining passive and doing nothing until the terminal time $T$, regardless of the performance of the bank or other agents' run. A more active agent may use the control $\alpha$ to anticipate her exit time depending on the values that she observes. In this way, $\alpha$ can be interpreted as the impatience of the agent, namely, the more negative the value the sooner the agent wants to leave the game. Other interpretations of the control variable $\alpha$ are also possible and depend on the model formulation. The classical socio-economics literature designed experiments to measure the level of trust among two parties. For example, in the famous \emph{trust game} (see \cite{Trust}), an individual is given a certain amount of initial capital and is asked to send part of it (or all of it) to a counterpart. The experimenter will double this amount and send it to the counterpart who needs to decide the portion of the received capital to send back to the first individual. If the individuals trust each others, both of them can obtain an amount equal to the initial capital. Through this experiment, the level of trust can be thus measured in a monetary amount.  In this example the bank may try to elicit the level of trust of its depositors which, also for strategic reasons, may not be directly observable. In a similar spirit of the above experiment, one could set up a questionnaire asking, for example, which percentage of their initial amount they would be willing to deposit in another bank if they had the possibility. 
	
In our model the trust level can also be influenced by different factors:
\begin{itemize}
	\item The performance of the bank, by means of the value of its assets $S_t$.
	\item The percentage of individuals who already left, by means of the value $1-L^N_t$. It is well documented that bank runs are often caused by panicking which are triggered by early withdrawals of other individuals (see among others \cite{BankRun1,BankRun2}). The idea is that an agent prefers to run before everyone else does.
	\item A private noise component $W^i_t$. The agent might want to withdraw early for private reasons such as unforeseen expenses or because of some non-rational components that are not modeled here. This noise component may also be used to describe the case where an agent cannot perfectly monitor the financial stability of the bank or the amount of depositors who already ran. One possibility is to let $\trust^i=(\hat{\trust}^i,Z^i)$ where $Z^i_t=L^N_t+\sigma^Z W^i_t$ and the trust process $\hat{\trust}^i$ depends on $Z^i$ rather than on $L^N$ directly.
\end{itemize}
\end{itemize}

\begin{example}\label{ex:bankrun}
 Let $b,\lambda,\mu,\sigma>0$ and $\gamma:\R\to\R$ be a bounded function which is positive on $\R_+$ and negative otherwise. 
 \begin{eqnarray*}
 	dS_t &=& \mu dt + \sigma^0 dW^0_t+e^{rt}dL^N_t\\
 	d\trust^i_t&=&(\alpha_t +\gamma (S_t-b) - \lambda L^N_t)dt+\sigma^\trust dW^i_t.
 \end{eqnarray*}
 In this example, the level of trust decreases if the value of the bank is below a certain threshold $b>0$ and increases otherwise. Trust diminishes also if some other depositor runs, the magnitude of such an effect can be modulated by the coefficient $\lambda>0$. 
\end{example}

The system \eqref{eq:model} can be written in a more compact form using a multidimensional process $X$ with exit time from a generic set $\cO$ as we will do in the rest of the paper. This allows to write Example \ref{ex:bankrun} with a single vector $X^i:=[S,\trust^i]$ with the complementary of $\cO=\R_+\times\R_+$ as absorbing set. 
We use modeling bank runs as a motivation to study the more general class of mean field games with absorption and common noise. Under some assumptions, we obtain existence of equilibria for the mean field model and approximate equilibria for the $N$-player game which apply, in particular, to the above model. A detailed study of the equilibria in models of bank run is an interesting question which is beyond the scope of this paper.

\section{MFG with common noise and absorption}\label{sec:MFG}
We start by introducing the pre-limit model.
Let $A\subset \R^k$ be a compact set representing the possible values of the admissible controls and $\cO\subset\R^d$ be an open set so that $\cO^C$ represents the absorbing region. The coefficients of private states' dynamics are given by the measurable functions $b:[0,T]\times \R^{d} \times \R\times A\to\R^{d}$, $h:\R^{d}\to\R$, $\eta:[0,T]\to\R^d$, the matrices $\sigma,\sigma^0\in \R^{d\times d}$ and a smoothing kernel $k:\R_+\to\R_+$ as in Definition \ref{def:kernel}. The objective functional to be maximized by each player is composed of measurable functions $f:[0,T]\times \R^{d} \times \R\times A\to\R$ 
and $G:[0,T]\times \R^{d}\to\R$, where $G$ is interpreted as a payoff which is received by the agent upon leave, whereas $f$ is interpreted as a running cost and therefore takes usually non-positive values.

\begin{assumption}\label{ass:alltogether}
	 The following are standing assumptions throughout the paper.
	\begin{enumerate}[label=\textnormal{\textbf{(H\arabic*)}},ref=(H\arabic*)]
		\item $b$, $f$, $h$ and $\eta$ are continuous and bounded on their domains. Moreover, $h$ is Lipschitz continuous. \label{hp:cont_bdd}
		\item $b$ is affine in the variable $a$, i.e., $b=b_1(t,x,m)+b_2(t)a$ for some bounded deterministic $b_2$. \label{hp:affine} 
		\item $\sigma,\sigma^0$ are constant matrices of full rank. \label{hp:sigmaconst}
		\item $f(t,x,m,a)$ is strictly concave in $a$ for every $(t,x,m)$ fixed. $G$ is bounded on $[0,T]\times\R^{d}$.\label{hp:terminal}
		\item The boundary of $\cO$ satisfies $\lambda_d (\partial \cO) =0$, where $\lambda_d $ denotes the Lebesgue measure on $\R^d$.\label{hp:boundary}
	\end{enumerate}
\end{assumption}

\begin{remark} Sufficient conditions granting the property (H5) above are, for instance, $C^2$ regularity of the boundary or the convexity of the set $\cO$.\end{remark}

The state dynamics for each player $i\in\{1,\ldots,N\}$ are constructed in a weak sense as follows. Let $(\Omega,\cF,\P)$ be a probability space, supporting a sequence of i.i.d. initial states $\{\xi^i\}_{i \geq 1}$ with common distribution with support in $\cO$, an infinite sequence of independent $\F$-Brownian motions $\{W^i\}_{i \geq 1}$ and the Brownian motion $W^0$. We assume that $\{\xi^i\}_{i \geq 1}$, $\{W^i\}_{i \geq 1}$ and $W^0$ are independent. The random variable $\xi^i$ represents the initial condition of player $i$, while the Brownian motions $W^i$ and $W^0$ represent, respectively, the idiosyncratic  and the common noise. We equip this probability space with the minimal filtration $\F^N :=(\cF^N _t)_{t\in[0,T]}$ generated by the first $N$ initial conditions $\{\xi^i\}_{i=1}^N$, the Brownian motions $\{W^i\}_{i=1}^N$ and the common noise $W^0$, and satisfying the usual conditions.
Let us start from the uncontrolled processes $X^i=(X^i_t)_{t\in[0,T]}$, $i=1,\ldots, N$, given by
	\begin{equation}\label{eq:Xi_weak_formulation}
	X^i _t=\xi^i +\sigma W^i _t+\sigma^0 W^0_t,\qquad t\in[0,T].
	\end{equation}	
Next, define 
\[\smL^N _t :=\int_0^tk(t-s) L^N _s ds = \int_0^tk(t-s) \langle 1, \mu^N _s \rangle ds,\] where $\mu^N_t$ is the empirical sub-distribution of players who have not been absorbed, i.e.,
\[
\mu^N_t=\frac{1}{N}\sum_{i=1}^N\delta_{X^i_t}(\cdot) 1_{[0,\tauex^i)}(t),\qquad\tauex^i:=\inf\{t\in[0,T]\ :\  X^i_t\notin \cO\},
\]
with the convention $\inf \emptyset = T$. Note that $\smL^N _t$ is $\P$-a.s.\ differentiable in $t$. Indeed, using the smoothness of $k$ and the a.e. continuity of $t\mapsto L_t ^N$ due to the function $t \mapsto 1_{\{\tauex^i >t\}}$ being continuous everywhere outside the event $\{\tauex^i =t\}$, which has zero measure under $\P$ (as in \cite[Lemma 6.2]{CF18}), we obtain
	\[{\smL^N_t}'=\frac{d}{dt}\int_0^tk(t-s)\langle 1,\mu^N _s\rangle ds=k(0)\langle 1,\mu_t\rangle+\int_0^tk'(t-s)\langle 1,\mu^N _s\rangle ds\quad \P\text{-a.s.}
	\]
	Since $d\smL^N_t = {\smL^N _t}' dt$ we define $\tilde{b}:[0,T]\times \R^{d} \times \cC([0,T];\subp(\R^{d}))\times A\to \R^{d}$ as 
	\begin{equation}\label{eq:btilde}
	\tilde{b}(t,x,m,a):=b(t,x,\langle h,m_t\rangle,a)+k(0)\langle 1,m_t\rangle+\int_0^tk'(t-s)\langle 1, m_s\rangle ds.
	\end{equation}
	Note that $\tilde{b}$ is uniformly bounded by \ref{hp:cont_bdd} and by the compactness of the support of $k$. The set of admissible controls $\cA_N$ for any player $i$ consists of all $A$-valued $\F^N$-progressively measurable processes. 
	For any $\alpha^i \in\cA_N$,  consider the stochastic exponential 
    \[
    \sexp(U^N)_t,\quad\text{with}\quad U^N_t:=\sum_{i=1}^N \int_0^t\sigma^{-1}\tilde{b}(s,X^i _s,\mu ^N ,\alpha^i _s)dW^i_s,
   \]
   which is a martingale since $\sigma^{-1}\tilde{b}$ is bounded. Define the measure $\P^{N,\alpha}\sim\P$ by 
   $d\P^{N,\alpha}/d\P:=\sexp(U^N)_T$. Girsanov's Theorem ensures that $(W^0,W^{1,\alpha^1}, \ldots, W^{N,\alpha^N})$ is a Brownian motion, where,
   \[ 
   W^{i,\alpha^i}_t:=W_t-\int_0^t\sigma^{-1}\tilde{b}(s,X^i_s,\mu^N,\alpha^i_s)ds,
   \]
leading to the following dynamics for the state variable under the new probability measure $\P^{N,\alpha}$

\begin{eqnarray}\label{eq:model2}
dX^i_t&=&b(t,X^i_t,\langle h,\mu^N_t\rangle,\alpha^i_t)dt+\sigma dW^{i,\alpha^i}_t+\sigma^0 dW^0_t +\eta(t)d\smL^N_t\\
X^i_0&=&\xi^i .\nonumber
\end{eqnarray}
The objective functional that each player aims to maximize is
\begin{equation}\label{eq:model2_cost}
J^i(\alpha)=\E^{N,\alpha}\bigg[\int_0^{\tauex^i} f(t,X^i_t,\langle h,\mu^N_t\rangle,\alpha^i_t)dt+G(\tauex^i,X^i_{\tauex^i})\bigg],
\end{equation}
where $\alpha = (\alpha^1, \ldots, \alpha^N)$ with $\alpha \in \cA^N$ for all $i=1,\ldots, N$. Notice that we chose the same function $h$ as in the drift of \eqref{eq:model2} only for the sake of simplicity. 

We aim at finding approximate Nash equilibria, which we define as follows. First, we denote by $[\alpha^{-i},\beta]$ the vector $\alpha$ where, in position $i$, the control $\alpha^i$ is replaced by the control $\beta$.

\begin{definition}\label{def:eps_Nash}Given $\varepsilon>0$, an $N$-dimensional vector of admissible controls $\alpha=(\alpha^1,\ldots,\alpha^N) \in \cA_N ^N$ is an $\varepsilon$-Nash equilibrium for the $N$-player game if for any admissible control $\beta$ and any $i\in\{1,\ldots,N\}$,
	\[ J^i(\alpha)\ge J^i([\alpha^{-i},\beta])-\varepsilon.
	\] 
\end{definition}
This notion of equilibrium is standard in the mean field games literature. The idea is that even if an agent would deviate from the equilibrium strategy $\alpha$, then the value function would increase of no more than $\varepsilon$, which we think small. Moreover, since the players are allowed to use any $\F^N$-progressively measurable process, we are looking at Nash equilibrium in open loop strategies. 

The above formulation is very convenient in order to derive the limiting mean field equations. Consider a filtered probability space $(\Omega,\cF,\F:=(\cF_t)_{t\in[0,T]},\P)$ satisfying the usual conditions and supporting independent $(\xi,W^0,W)$, where $\xi\in L^2(\Omega,\cF_0,\P)$ is the initial condition and  $W^0,W$ are $d$-dimensional Brownian motions. A random flow of sub-probabilities $\mu$ defined as
\[
(t,\omega)\mapsto \mu_t(\omega)\in \subp(\R^{d})
\]
represents the limit of $(\mu^N_t)_{t\in[0,T]}$ as $N\to\infty$. Differently from the case without common noise, this flow is not deterministic as the common source of noise does not disappear in the limit. Moreover, the fact that $\mu_t$ is a sub-probability represents the absorption effect. We recall that $\subp(\R^{d})$ is equipped by the Borel $\sigma$-field generated by the topology of weak convergence of measures. 
If, for a moment, we suppose that the flow $\mu$ is given and fixed we can compute the optimal response of an agent to the given flow. For a certain control process $\alpha$, the state dynamics reads as 
\begin{eqnarray}\label{eq:MFG_state}
dX_t&=&b(t,X_t,\langle h,\mu_t\rangle,\alpha_t)dt+\sigma dW_t+\sigma^0  dW^0_t +\eta(t)d\smL_t\nonumber\\
X_0&=&\xi
\end{eqnarray}
where  $\smL_t:=\int_0^tk(t-s)L_s ds$ and $L_t=\langle 1,\mu_t\rangle$.
The objective functional to maximize is 
\begin{equation}\label{eq:MFG_cost}
J(\alpha)=\E\bigg[\int_0^{\tauex} f(t,X_t,\langle h,\mu_t\rangle,\alpha_t)dt+G(\tauex,\Xtauex)\bigg],
\end{equation}
with $\tauex:=\inf\{t\in[0,T]\ :\  X_t\notin \cO\}$, over the set of $A$-valued progressively measurable processes called admissible controls and denoted by $\cA$. 
The solution to the optimization problem is understood in the weak sense, more details including its rigorous formulation will be given at the beginning of Section \ref{sec:FBSDEopt}. When a maximum exists, an optimal control $\hat{\alpha}$ generates a new flow of conditional sub-probabilities 
\begin{equation}\label{eq:subprob}
\hat \mu_t(B):=\P\big(\{X^{\hat{\alpha}}_t\in B\} \cap \{\tauex>t\}\mid W^0\big),\qquad B\in\cB(\R),
\end{equation}
where $X^{\hat{\alpha}}$ is the solution to \eqref{eq:MFG_state} with $\hat{\alpha}$ as control.
A strong mean field game equilibrium is obtained as a fixed point of the above procedure. 
\begin{definition}\label{def:MFGequilibrium}
A filtered probability space  $(\Omega, \cF, \F, \P)$ supporting independent $(\xi,W^0,W)$ together with a random flow of sub-probabilities $(t,\omega)\mapsto \mu_t(\omega)\in \subp(\R^{d})$ and a control policy $\alpha$ is called a \emph{strong MFG equilibrium} if
\begin{itemize}
\item $\alpha$ is an optimal control for the optimization problem \eqref{eq:MFG_state},\eqref{eq:MFG_cost};
\item $\mu_t(\cdot)=\P\big(\{X^\alpha_t\in \cdot\} \cap \{\tauex>t\}\mid \cF_t^{W^0}\big)$, where $X^\alpha$ is the solution to \eqref{eq:MFG_state} for the given control $\alpha$ and $(\cF_t^{W^0})_{t\in[0,T]}$ is the filtration generated by $W^0$.
\end{itemize}
\end{definition}
We stress that throughout the paper the probability space is part of the solution to the MFG problem so all the solutions are weak in the probabilistic sense. We use instead the adjectives strong and weak in relation to the properties of the MFG equilibria. Indeed, adopting the terminology of \cite{CDL16} we call \emph{strong} those equilibria for which the conditioning in the equilibrium distribution is with respect to the sole common noise $W^0$, enlarged with the initial condition $\xi$ if needed. Despite natural, strong solutions are difficult to obtain under general assumptions. We thus consider a \emph{weak} version of the above problem where another filtration $\F$, possibly larger than the one generated by $W^0$, dictates the second condition of Definition \ref{def:MFGequilibrium}. Since $\F$ is not necessarily the one generated by the common noise it is important to impose conditions under which we find sensible solutions, namely, we need to guarantee that the additional source of randomness contained in $\F$ does not provide future information on the noise of the system. We achieve this with the help of the immersion property: a filtration $\G\subset \F$ is \emph{immersed} in $\F$ if every square-integrable $\G$-martingale is a square-integrable $\F$-martingale. 
\begin{definition}\label{def:MFGequilibrium_weak}
	A filtered probability space  $(\Omega, \cF, \F, \P)$ supporting independent $(\xi,W^0,W)$ together with a random flow of sub-probabilities $(t,\omega)\mapsto \mu_t(\omega)\in \subp(\R^{d})$ and a control policy $\alpha$ is called a \emph{weak MFG equilibrium} if
	\begin{itemize}
		\item $\alpha$ is an optimal control for the optimization problem \eqref{eq:MFG_state},\eqref{eq:MFG_cost};
		\item $\mu_t(\cdot)=\P\big(\{X^\alpha_t\in \cdot\} \cap \{\tauex>t\}\mid \cF_t^{W^0,\mu}\big)$, where $X^\alpha$ is the solution to \eqref{eq:MFG_state} for the given control $\alpha$ and $(\cF_t^{W^0,\mu})_{t\in[0,T]}$ is the filtration generated by $(W^0,\mu)$;
		\item The filtration generated by $(\xi,W,W^0,\mu)$ is immersed in $\F$.
	\end{itemize}
\end{definition}
The first two conditions are analogous to those of Definition \ref{def:MFGequilibrium}, the third one guarantees that the enlargement of the filtration still produces sensible solutions. In the next section 
we discuss more the notion of weak equilibrium and describe how to obtain such solutions. In particular, we prove the following main existence result for mean field games with absorption and common noise.
\begin{theorem}\label{thm:existence}
	Under Assumption \ref{ass:alltogether}, there exists a weak MFG equilibrium.
\end{theorem}
Within this framework, we can construct a mean field model of bank run for $X:=[S,\trust]$ representing the bank value and the trust dynamics, as described in Section \ref{sec:motivation}. We assume $S_0,\trust_0>0$. Moreover, we implicitly assume that that all agents share the same $b^i=b$, $f^i=f$ and $G^i=G$, in order to make the mean field approximation meaningful. The terminal reward function is given $G(\tauex,X_{\tauex}):=(e^{r\tauex}D_0)\wedge S_{\tauex}$, where $D_0>0$ is the initial deposit, $r>0$ the interest rate and $\tauex$ is the minimum between $T>0$ and the exit time from the set $\cO=\R_+\times\R_+$, i.e., the first time either the agent runs or the bank defaults. Note that with this choice of $G$ and $\cO$, assumptions \ref{hp:terminal}-\ref{hp:boundary} are satisfied. For the other coefficients, we require them to satisfy  Assumption \ref{ass:alltogether}. Therefore, the following result is straightforward.
	\begin{corollary} Under  Assumption \ref{ass:alltogether}, the above mean field model of bank run admits a weak MFG equilibrium.
	\end{corollary}

\section{Existence of weak mean field equilibria}\label{sec:weakeq}
One major difficulty in studying the MFG problem of Section \ref{sec:MFG} is that, due to the presence of a common noise, the flow of sub-probabilities $\mu$ is a stochastic process and some care is required for handling its measurability properties. In absence of common noise, the existence of an equilibrium is given by a fixed point argument yielding a \emph{deterministic} flow of sub-probability. To that aim one can use well known compactness criteria which are, in our framework, no longer available. To cope with this difficulties, we proceed in the spirit of \cite{CDL16} and solve first a discretized problem and then pass to the limit. The discretization consists in constraining the flow $\mu$ to be measurable with respect to a finite filtration induced by $W^0$. By construction, this approximated problem admits a strong solution in the sense of Definition \ref{def:MFGequilibrium}, however, the passage to the limit does not preserve this property and therefore only a weak equilibrium is obtained.  

A key ingredient of the analysis is represented by a lifted environment $\env=(\fM,\fA)$. The first component will represent the conditional law of the process $(X,W)$, with $X$ the equilibrium state process. The second component will represent the conditional law of the optimal control $\alpha$ seen as a relaxed control (see equations \eqref{eq:equilibrium} below). 
It is important to point out that $\env$ will not be only a technical tool for performing the limiting procedure but it will also  determine the filtration $\F$ for the existence of a weak MFG equilibrium and, in turn, the information available to the agents for their optimization task. Indeed, given $\env$, we construct the process $(\env_t)_{t\in [0,T]}$, where $\env_t$ is the restriction of the probability measures $\fM$ and $\fA$ to $\cC([0,t];\R^{2d})$ and $[0,t]\times A$, respectively.
\begin{definition}\label{def:natural}
We call the \emph{natural filtration} of $\env$, the sigma-algebra $\sigma(\env_s \mid 0\le s\le t)$. The same terminology will be used for its component $\fM$ and $\fA$ when needed.	
\end{definition}
We next define the technical setup in which we work to prove Theorem \ref{thm:existence} and to guarantee that we obtain sensible weak MFG equilibria as in Definition \ref{def:MFGequilibrium_weak}.

\begin{definition}\label{def:admissiblesetup} Let $(\Omega^i,\cF^i,\F^i,\P^i)$ for $i=0,1$ be filtered probability spaces satisfying the usual assumptions and supporting two (independent) Brownian motions $(W^0,W=W^1)$ and an independent initial condition $\xi$ on $\Omega^1$. Let $\env:\Omega^0\to\prob(\cC([0,T];\R^{2d}))\times\prob(\cM([0,T]\times A))$ be an $\cF^0_T$-measurable random variable. We construct the product space $\Omega:=\Omega^0\times\Omega^1$ endowed with $\cF:=\cF^0\otimes\cF^1$ and use the measure $\P:=\P^0\otimes \P^1$ to complete the right-continuous product filtration $\F$ with the $\P$-null sets. The couple $(\Omega,\cF,\F,\P)$ and $\env$ is called an \emph{admissible setup} if the filtration generated by $(\xi,W^0,W,\env)$ is immersed in $\F$.
\end{definition}

\begin{definition}\label{def:liftedMFGeq} Let $(\Omega,\cF,\F,\P)$ and $\env$ be an admissible setup.
	We say that $\env=(\fM,\fA)$ induces a \emph{weak MFG equilibrium }if by setting,   
	\begin{equation}\label{eq:induced_flow}
	\mu_t(\omega^0)(B):=\int_{\cC([0,T];\R^{d})} 1_{B}(x_t)1_{\{\tau(x)>t\}}\fM^X(\omega^0)(dx),\quad \omega^0\in\Omega^0,\ B\in \cB_{\R^{d}},
	\end{equation}
	with $\tau(x):=\inf\{t\in[0,T]\ :\ x_t\notin \cO\}$ and $\fM^X$ the first marginal of $\fM$, there exists an $\F$-progressively measurable $\alpha$ with values in $A$ solving \eqref{eq:MFG_state}-\eqref{eq:MFG_cost} such that: 
	\begin{equation}\label{eq:equilibrium} \fM(\omega^0)=\law(X^\alpha(\omega^0,\cdot),W(\cdot))\quad \text{and}\quad \fA(\omega^0)=\law\left(\delta_{\alpha_t (\omega^0,\cdot)}dt \right),
	\end{equation}
for every $\omega^0$ outside a $\P^0$-null set, where $\delta_{\alpha_t (\omega^0,\cdot)}dt$ denotes the measure on $[0,T] \times A$ induced by the control $\alpha$.
When $\alpha$ is a solution of \eqref{eq:MFG_state}-\eqref{eq:MFG_cost} in the relaxed sense, i.e.
the associated relaxed control $\delta_{\alpha_t (\omega^0,\cdot)}dt$ maximizes 
\begin{equation}\label{eq:MFG_cost_relax}
J(\gamma)=\E\bigg[\int_0^{\tauex} f(t,X_t,\langle h,\mu_t\rangle,a)\gamma(dt, da)+G(\tauex,\Xtauex)\bigg],
\end{equation}
with $\tauex:=\inf\{t\in[0,T]\ :\  X_t\notin \cO\}$, over the set of $\F$-progressively measurable processes $\gamma$ with values in\[
\Gamma:=\left\{q\in\cM([0,T]\times A)\ :\ q(dt,da)=dt\ q_t(da),\text{ for some Borel measurable kernel }q_t\right\},
\] subject to the following relaxed state dynamics
\begin{eqnarray}\label{eq:MFG_state_relax}
dX_t&=&\int_Ab(t,X_t,\langle h,\mu_t\rangle,a)\gamma(dt, da) +\sigma dW_t+\sigma^0  dW^0_t +\eta(t)d\smL_t\nonumber\\
X_0&=&\xi
\end{eqnarray}
where  $\smL_t:=\int_0^tk(t-s)L_s ds$ and $L_t=\langle 1,\mu_t\rangle$, we say that $\env=(\fM,\fA)$ induces a \emph{relaxed weak MFG equilibrium}.
\end{definition} 
The condition \eqref{eq:equilibrium} describes the fact that $\env$ encodes all the relevant aspects for the equilibrium: it defines $\alpha$ and $\mu$ such that $\alpha$ is optimal for \eqref{eq:MFG_state}-\eqref{eq:MFG_cost} and $\mu$ is a fixed point. On the other hand the admissibility of the setup automatically implies the immersion property of Definition \ref{def:MFGequilibrium_weak}. It follows that  $(\Omega, \cF, \F, \P, \mu,\alpha)$ is a weak MFG equilibrium as in Definition \ref{def:MFGequilibrium_weak}  (see also the discussion of \cite[(Vol II) Chapter 2]{bookMFG} in a similar setup, in particular, Remark 2.19).

We adopt a probabilistic approach and reduce the problem to solving a suitable class of McKean-Vlasov FBSDEs.
As discussed above, it is implicit that we need to guarantee that a solution exists on an admissible setup and since the filtration $\F$ is not assumed to be the one generated by the Brownian motions, an extra orthogonal martingale term will appear. The equilibrium equation is of the following form: 
\begin{eqnarray}
dX_t&=&b(t,X_t,\ell^X_t,\hat{\alpha}_t)dt+\sigma dW_t+\sigma^0  dW^0_t +\eta(t)d\smL_t\nonumber\\
dY_t&=& f(t,X_t,\ell^X_t,\hat{\alpha}_t)dt + Z_tdW_t+Z^0_t dW^0_t+dM_t\label{eq:mkvFBSDE}
\end{eqnarray}
with boundary conditions $X_0=\xi$ and $Y_{\tauex}=G(\tauex,\Xtauex)$, where  $\tauex:=\inf\{t\in[0,T]\ :\  X_t\notin \cO\}$, $\hat{\alpha}$ is a suitable optimal control to be determined and 
\[
\ell^X_t=\E[h(X_t)1_{\tauex>t}\mid \cF^0_t].
\]
\begin{definition}
	A process $(X,Y,Z,Z^0,M)$ on a filtered probability space $(\Omega,\cF,\F,\P)$ is a (weak) solution to the McKean-Vlasov FBSDE \eqref{eq:mkvFBSDE} if it is $\F$-progressively measurable, it satisfies
	\[\E\left[\sup_{0\le t\le T}\big(|X_t|^2+|Y_t|^2+|M_t|^2\big)+\int_0^T \big(|Z_t|^2+|Z^0_t|^2\big) dt \right]<\infty, 
	\]
	the process $M$ is orthogonal to $(W,W^0)$ and equation \eqref{eq:mkvFBSDE} is satisfied.
\end{definition}

The proof of Theorem \ref{thm:existence} consists essentially in finding solutions to \eqref{eq:mkvFBSDE} and it is divided in two parts. First we find an equilibrium in a reduced setting where the common noise is only allowed to take finitely many values and where the admissibility condition in Definition \ref{def:admissiblesetup}  is automatic; this is achieved via a fixed point procedure based on Schauder's Theorem. Second we show how to construct a sequence of reduced equilibria, represented by some lifted environments $\env^n$, which converge to a suitable $\env^\infty$. The admissibility of the limiting environment will induce the desired MFG equilibrium.

\subsection{The FBSDE in a random environment}\label{sec:FBSDEopt}
In this section we are given an input filtered probability space $(\Omega,\cF,\F,\P)$ supporting two independent Brownian motions $(W^0,W)$ and an independent initial condition $\xi\in L^2(\Omega,\cF,\P)$ with distribution $\law(\xi)$. We suppose that the filtration is right-continuous and completed with the $\P$-null sets in $\cF$.
We also suppose that we are given an $\F$-progressively measurable input process $\mu:[0,T]\times\Omega\to\subp(\R^{d})$ with continuous trajectories $t\mapsto \mu_t (\omega)$ with respect to weak convergence, for all $\omega \in \Omega$. Since the measure-valued input process is fixed, the FBSDEs in this section will not be of McKean-Vlasov type.
We also do not assume that the filtration $\F$ is the one generated by the Brownian motion so that, in the following, we work with FBSDEs in a general space (see for example \cite{bookElkaroui1997backward}). We only assume here that $L^2(\Omega,\F,\P)$ is separable, an hypothesis that is satisfied by all our cases of interest.

In order to describe the weak formulation of the control problem \eqref{eq:MFG_state}-\eqref{eq:MFG_cost}, let us proceed as in for $N$-player game. Despite the great similarity, we give all details for reader's convenience. We start from the uncontrolled process $X=(X_t)_{t\in[0,T]}$ given by
	\begin{equation}\label{eq:X_weak_formulation}
	X_t=\xi+\sigma W_t+\sigma^0 W^0_t,\qquad t\in[0,T].
	\end{equation}	
	The set of admissible controls $\cA$ consists of all $A$-valued $\F$-progressively measurable processes. For any $\alpha\in\cA$,  consider the stochastic exponential 
    \[
    \sexp(U)_t,\quad\text{with}\quad U_t:=\int_0^t\sigma^{-1}\tilde{b}(s,X_s,\mu,\alpha_s)dW_s,
   \]
   with $\tilde{b}$ as in \eqref{eq:btilde}, which is a martingale since $\sigma^{-1}\tilde{b}$ is bounded by \ref{hp:cont_bdd} and by the compactness of the support of $k$. Define the measure $\P^{\mu,\alpha}\sim\P$ by 
   $d\P^{\mu,\alpha}/d\P:=\sexp(U)_T$. Girsanov's Theorem ensures that $(W^0,W^{\alpha})$ is a Brownian motion, where,
   $ 
   W^{\alpha}_t:=W_t-\int_0^t\sigma^{-1}\tilde{b}(s,X_s,\mu,\alpha_s)ds,
   $
leading to the following dynamics for the state variable
\[X_t = \xi + \int_0 ^t b(s,X_s,\langle h,\mu_s\rangle,\alpha_s)ds+\sigma W^\alpha_t+\sigma^0 W^0_t +\int_0 ^t \eta(s)d\smL_s,\]
under the measure $\P^{\mu,\alpha}$. The weak formulation of \eqref{eq:MFG_state}-\eqref{eq:MFG_cost} consists in the following control problem
\begin{equation}\label{eq:MFG_weak}\sup_{\alpha \in \mathcal A} J^{weak} (\alpha) := \sup_{\alpha \in \mathcal A} \E^{\alpha}\left[\int_{0}^{\tauex} f(t,X_t,\langle h,\mu_t\rangle,\alpha_t)dt+G(\tauex,\Xtauex)\right],\end{equation}
where we denote $\E^\alpha$ the expectations under $\P^{\mu,\alpha}$. 

In order to find an optimizer of the above problem, we consider the Hamiltonian of the system $H:[0,T]\times \R^{d} \times \cC([0,T];\subp(\R^{d}))\times \R^{d}\times A\to\R^{d}$, defined as
	\[
	H(t,x,m,z,a):=f(t,x,\langle h,m_t\rangle,a)+z\cdot\sigma^{-1}\tilde{b}(t,x,m,a),
	\]
	where the inverse of $\sigma$ exists by \ref{hp:sigmaconst}.
	Continuity of $H$ follows by \ref{hp:cont_bdd} and the fact that $k(0)l_t+\int_0^tk'(t-s)l_s ds$ is a continuous function of an $\R$-valued process $l$, with respect to the supremum norm. From the compactness of $A$, $H$ admits a maximum on $A$ for every fixed $\theta:=(t,x,m,z)\in[0,T]\times \R^{d} \times \cC([0,T];\subp(\R^{d}))\times\R^{d}$.
	Moreover, using  Berge's maximum theorem (see e.g.\ \cite[Theorem 17.31]{AliprantisBorder06}), our assumption guarantees that the $\hat{H}(\theta):=\max_{a\in A}H(\theta,a)$ is a continuous function of $\theta$ and there exists a continuous function $\hat{a}: [0,T]\times \R^{d} \times \cC([0,T];\subp(\R^{d}))\times\R^{d}\to A$ such that
		\begin{equation}\label{eq:existence max}
		\hat{H}(\theta)=\max_{a\in A}H(\theta,a)=H(\theta,\hat{a}(\theta)),\quad \forall\theta\in [0,T]\times \R^{d} \times \cC([0,T];\subp(\R^{d}))\times\R^{d}.
		\end{equation}
  Such a maximizer is unique by the strict concavity of $f$ in \ref{hp:terminal}. 

\begin{proposition}\label{prop:exFBSDE}
	Let $\mu=(\mu_t)_{t\in[0,T]}$ be a given $\F$-progressively measurable process with values in $\subp(\R^{d})$, such that all paths are continuous. Let $X$ be given by \eqref{eq:X_weak_formulation} and let $(\hat Y, \hat Z, \hat Z^0, \hat M)$ be the unique solution of the BSDE
\begin{eqnarray}
dY_t&=& \hat H\big(t,X_t, \mu, \hat a(t, X_t, \mu, Z_t) \big)dt +Z_tdW_t+Z^0_t dW^0_t+dM_t\nonumber\\
Y_{\tauex}&=&G(\tauex,\Xtauex),\label{eq:randomFBSDE}
\end{eqnarray}
with $\tauex:=\inf\{t\in[0,T]\ :\  X_t\notin \cO\}$. Then, the admissible control $\hat{\alpha}=(\hat{\alpha}_t)_{t\in[0,T]}$ given by $\hat{\alpha}_t:=\hat{a}(t,X_t,\mu,\hat{Z}_t)$ is an optimizer of the control problem \eqref{eq:MFG_weak}.
\end{proposition}
\begin{proof}

	In the following we will seek for solutions $(Y,Z,Z^0,M)$ to BSDEs of the form
    \begin{eqnarray*}
	Y_{t\wedge \tauex} &=& G(\tauex,\Xtauex)+\int_{t \wedge \tauex}^{\tauex} \Psi(\omega,s,Z_s)ds\\
	&&-\int_{t \wedge \tauex}^{\tauex} Z_s dW_s-\int_{t\wedge \tauex}^{\tauex} Z^0_s dW^0_s -\int_{t\wedge \tauex}^{\tauex} dM_s,
	\end{eqnarray*}
	with drivers $\Psi(\omega,s,z)$ which are Lipschitz continuous in $z$ uniformly in $(\omega,s)$ and bounded as functions of $(\omega,s)$ for every $z$ fixed.
	Recall that the filtration is not assumed to be Brownian, thus an extra martingale term $M$ orthogonal to $(W,W^0)$ is part of the solution.
	For any of such drivers classical results on BSDE, collected for convenience in Theorem \ref{thm:BSDEapp} in the Appendix, guarantee existence and uniqueness of the solution $(Y,Z,Z^0,M)$.
	We choose
	\begin{itemize}
		\item $\Psi(\omega,s,z)=\hat{H}(s,X_s(\omega),\mu(\omega),z)$ with $\hat{H}$ defined in \eqref{eq:existence max};
		\item $\Psi(\omega,s,z)=f(s,X_s,\langle h,\mu_s\rangle,\alpha_s)(\omega)+z\cdot\sigma^{-1}\tilde{b}(s,X_s,\mu,\alpha_s )(\omega)$ with $\alpha$ an admissible control.		
	\end{itemize}
	Both satisfy the required hypothesis and we call $(\hat{Y},\hat{Z},\hat{Z}^0,\hat{M})$ and $(Y^\alpha,Z^\alpha,Z^{0,\alpha},M^\alpha)$ the respective solutions.
    The comparison result of Theorem \ref{thm:BSDEapp} and the definition of $\hat{H}$ yield $\hat{Y}\ge Y^\alpha$ $\P$-a.s.\ for any $\alpha\in\cA$. 
    We now choose the control $\hat{\alpha}_t :=\hat{a}(t,X_t,\mu,\hat{Z}_t)$, where $\hat{a}$ is the continuous function realizing the pointwise maximum of the Hamiltonian as discussed before \eqref{eq:existence max}. By \eqref{eq:existence max}, it holds 
    \[
    \hat{H}(t,X_t,\mu,\hat{Z}_t)=f(t,X_t,\langle h,\mu_t\rangle,\hat{\alpha}_t)+\hat{Z}_t \cdot\sigma^{-1}(t,x)\tilde{b}(t,X_t,\mu,\hat{\alpha}_t )\quad \P\text{-a.s.}
    \]
    so that $(Y^{\hat{\alpha}},\hat{Z},\hat{Z}^0,\hat{M})$ is a solution to the above BSDE with driver $\Psi(\omega,z)=\hat{H}(t,X_t(\omega),\mu(\omega),z)$, i.e. BSDE \eqref{eq:randomFBSDE}.
    By uniqueness of the solution, it must hold 
    \begin{equation}\label{eq:maxZ}
    Y^{\hat{\alpha}}_t=\hat{Y}_t\ge Y^\alpha_t \ \P\text{-a.s.}\qquad \forall t\in [0,T],\ \forall \alpha\in\cA.
    \end{equation}
      Notice that since $\xi$ and $W$ are independent, $\xi$ has the same distribution under both $\P$ and $\P^{\mu,\alpha}$. With respect to the measure $\P^{\mu,\alpha}$, the process  $(X,Y^\alpha,Z^{0,\alpha},Z^\alpha)$ solves the FBSDE
    \begin{eqnarray*}
    	X_t&=& \xi + \int_0 ^t b(s,X_s,\langle h,\mu_s\rangle,\alpha_s)ds+\sigma W^\alpha_t+\sigma^0 W^0_t +\int_0 ^t \eta(s)d\smL_s\\
   	Y^\alpha_{t\wedge \tauex} &=& G(\tauex,\Xtauex)+\int_{t \wedge \tauex}^{\tauex}f(s,X_s,\langle h,\mu_s\rangle,\alpha_s)ds -\int_{t \wedge \tauex}^{\tauex} Z^{\alpha}_s dW^{\alpha}_s-\int_{t \wedge \tauex}^{\tauex} Z^{0,\alpha}_s dW^0_s -\int_{t\wedge \tauex}^{\tauex} dM^{\alpha}_s,
   \end{eqnarray*}
Recall that $\E$ and $\E^\alpha$ denote the expectations under $\P$ and $\P^{\mu,\alpha}$, respectively.
Since the last three terms in the r.h.s. of the equation for $Y^\alpha$ are $\P^{\mu,\alpha}$-martingales,  by taking conditional expectation on both sides, we obtain
\[
Y^\alpha_{t \wedge \tauex}=\E^{\alpha}\left[\int_{t\wedge \tauex}^{\tauex} f(s,X_s,\langle h,\mu_s\rangle,\alpha_s)ds+G(\tauex,\Xtauex)\mid\cF_{t\wedge \tauex}\right].
\]
Since $Y^\alpha (0)$ is $\cF_0$-measurable and $\cF_0$ and $W$ are independent, $Y^\alpha (0)$ has the same law under both $\P$ and $\P^{\mu,\alpha}$, so that we have $\E[Y^\alpha(0)]=\E^{\alpha}[Y^\alpha(0)]$. Moreover, the latter is equal to the payoff function in the weak formulation, namely,
 $\E^{\alpha}[Y^\alpha(0)]=\E^{\alpha}[\int_{0}^{\tauex} f(t,X_t,\langle h,\mu_t\rangle,\alpha_t)dt+G(\tauex,\Xtauex)]=J^{weak}(\alpha)$.
 In particular, the above applies to the control $\hat{\alpha}$ constructed above so that, combining with \eqref{eq:maxZ},
 \[J^{weak}(\hat{\alpha})=\E[Y^{\hat{\alpha}}(0)]\ge \E[Y^\alpha(0)]=J^{weak}(\alpha),
 \]
 which shows that $\hat{\alpha}$ is optimal for the control problem \eqref{eq:MFG_weak}. 
\end{proof}

\subsection{The fixed point argument}\label{sec:fixedpoint}
In this section we implement the fixed point argument restricted to conditional distribution flows taking only finitely many values.
The starting point is again a process $\tilde{X}$, defined on some probability space supporting independent Brownian motions $(\tilde{W},\tilde{W}^0)$ and $\tilde{\xi}$ distributed as $\xi$ from Section \ref{sec:FBSDEopt}, satisfying \begin{equation}\label{eq:basic_proc}
\tilde{X}_t=\tilde{\xi}+\sigma \tilde{W}_t+\sigma^0 \tilde{W}^0_t, \quad t\in[0,T].
\end{equation}
Since we adopt a weak approach, we will work on a suitable canonical space and the only important feature of the processes in \eqref{eq:basic_proc} will be their laws. More specifically, we consider a filtered probability space $(\Omega,\cF,\F,\P)$ of the product form as in Definition \ref{def:admissiblesetup}. For both $i=0,1$ we use the following specifications: $\Omega^i:=\cC([0,T];\R^{d})\times\cC([0,T];\R^{d})$ is the state space and $\F^i$ is the natural filtration generated by the canonical process on $\Omega^i$; for $\P^i$ we take the joint distribution of a process $(\tilde{X},\tilde{W})$ as in \eqref{eq:basic_proc}; $\cF^i$ is the completion of the Borel $\sigma$-algebra of $\Omega^i$ with respect to $\P^i$. We now construct the processes $(X,W,W^0)$ which are relevant to our analysis. 

\begin{itemize}
	\item The pair $(X,W)$ is simply the canonical process on $\Omega^1$. By construction of $\P$, $W$ is a $d$-dimensional Brownian motion independent of $X_0\sim\xi$.
	\item The process $W^0$ is defined on $\Omega^0$ as 
	\[W^0_t(x,w):=(\sigma^0)^{-1}(x_t-x_0-\sigma w_t), \quad (x,w) \in \Omega^0, t\in[0,T].
	\]
	
	By construction of $\P$, $W^0$ is a $d$-dimensional Brownian motion, which is further independent of $(X_0,W)$.
\end{itemize}
We suppose that a partition of the time interval $\{0=t_0<t_1<\cdots< t_{N-1}<t_N=T\}$ and a finite set $\Lambda\subset \R^{d}$ are given.
In addition, we have a process on $\Omega^0$ of the form 
\[V_t(\omega^0)=\sum_{i=1}^{N} v_{i-1}(\omega^0)1_{[t_{i-1},t_i)}(t)+v_{N-1}(\omega^0)1_{\{T\}}(t),\qquad \omega^0\in\Omega^0,
\]
that we suppose adapted to the filtration generated by $W^0$ and such that each $v_{i}$ takes values only in the given finite set $\Lambda$, i.e. $v_{i}: \Omega^0 \to \Lambda$ is $\mathcal F^0 _{t_i}$-measurable, for every $i=0,\ldots, N-1$. We let $\cV:=\{A_1,\ldots,A_{|\cV|}\}$ be the algebra generated by $V$ and suppose that $\P^0(A_k)>0$ for every $k=1,\ldots,|\cV|$. The process $V$ represents a discretization of the Brownian Motion $W^0$; at this stage we only need the above properties, we will consider a specification of $V$ in Section \ref{sec:tight}.

Let $M:=\exp(c^2T)$ where $c$ is an upper bound for $|\sigma^{-1}\tilde{b}|$ with $\tilde{b}$ as in \eqref{eq:btilde}. Note that this exists by \ref{hp:cont_bdd}, \ref{hp:sigmaconst} and Definition \ref{def:kernel} and it depends only on the coefficients of the problem.  Define,
\begin{equation}\label{eq:compact_measures}
\fxspace:=\left\{\P'\in\prob(\Omega^1)\ :\ \P'\ll \P^1,\ \E\left[\big(d\P'/d\P^1\big)^2\right]\le M\right\}.
\end{equation}
We endow $\fxspace$ with the so-called $\tau$-topology, i.e. weakest topology that makes $f\mapsto \int f d\P'$ continuous, for every bounded measurable function $f$ (cf. \cite[Section 6.2]{dembo2010}).
As in the proof of \cite[Proposition 7.8]{CL15}, the set $\fxspace$ is convex, compact and metrizable with respect to the $\tau$-topology.

Suppose that a certain vector of measures $(\fm_1,\ldots,\fm_{|\cV|}) \in \fxspace^{|\cV|}$ is given. We consider $\fM:\Omega^0\to \prob(\Omega_1)$ defined by $\fM(\omega^0):=\sum_{k=1}^{|\cV|}\fm_k1_{A_k}(\omega^0)$ and note that its natural filtration, in the sense of Definition \ref{def:natural}, is contained in the natural filtration of $W^0$. Finally, we construct the random environment $\mu:[0,T]\times\Omega^0\to\subp(\R^{d})$ as prescribed by \eqref{eq:induced_flow}. Observe that $t \mapsto \mu_t (\omega^0)$ is continuous for all $\omega^0 \in \Omega^0$. Indeed, notice first that the function $t \mapsto 1_{\{\tau(x) >t\}}$ is continuous everywhere outside the event $\{\tau(x)=t\}$, which has zero measure with respect to $\mathfrak M^X (\omega^0)$ for all $\omega^0 \in \Omega^0$ (as in \cite[Lemma 6.2]{CF18}). This implies that, for every continuous and bounded function $f$, we have $\int f(x_t)1_{\{\tau(x)>t\}} \mathfrak M^X (\omega^0)(dx)$ is continuous in $t$ for all $\omega^0 \in \Omega^0$.

We can then apply the results of Section \ref{sec:FBSDEopt} to obtain a solution to the FBSDE \eqref{eq:randomFBSDE} on $(\Omega,\cF,\F,\P^\mu)$, for a suitable $\P^\mu=\P^{\mu,\hat{\alpha}}\sim \P$ given by Proposition \ref{prop:exFBSDE}.  
From the measurability of $\fM$, the solution is of the form $(Y,Z,Z^{0})$, i.e. $M\equiv0$. 
We then construct the vector of conditional distributions of $(X,W)$ given the events in $\cV$:
\[ 
\fm'_k(\cdot):=\frac{\P^{\mu}(A_k\times(X,W)^{-1}(\cdot) )}{\P^\mu(A_k\times \Omega^1)},\quad k=1,\ldots, |\cV|.
\]
Recall that the change of measure of Proposition \ref{prop:exFBSDE} is obtained by a Girsanov transformation of $W$ only. In particular, $W^0$ is still a Brownian motion under $\P^\mu$ and it holds $\P^\mu(A_k\times\Omega^1 )=\P^0(A_k)>0$ for every $\mu$. The above definition is thus well-posed and, since $(X,W)$ is the canonical process on $\Omega^1$, we may write
\[ 
\fm'_k(\cdot)=\frac{\P^{\mu}( A_k\times\cdot )}{\P^0(A_k)},\quad \forall\ k=1,\ldots, |\cV|.
\]
We now prove that $\fm'_k\ll \P^1$ for every $k$ and that its Radon-Nikod\'ym derivative is
\[\hat{Z}_k:=\frac{1}{\P^0(A_k)}\int_{\Omega^0} 1_{A_k}\frac{d\P^\mu}{d\P}d\P^0,\quad k=1,\ldots,|\cV|.\]
By its definition and the fact that $\P = \P^0 \otimes \P^1$, it is easy to see that $\fm'_k(B)=\E^{\P^1}[\hat{Z}_k1_B]$ for every $B\in\cF^1$ so we only need to prove $\hat Z_k$'s integrability, for all $k =1,\ldots, |\cV|$.
We prove something stronger, namely, that $\hat Z_k$ is square-integrable with second moment uniformly bounded by $M$ from which, in particular, $\fm'_k\in\fxspace$ as in \eqref{eq:compact_measures}.
To simplify the notation we introduce the measures $\P^{0,k}(\cdot):=\P^0(\cdot\cap A_k)/\P^0(A_k)$, so that $\hat{Z}_k=\int_{\Omega^0} \frac{d\P^\mu}{d\P} d\P^{0,k}$.
By applying Jensen's inequality and Fubini's Theorem, we get,
\[
\int_{\Omega^1}(\hat{Z}_k)^2 d\P^1\le\int_{\Omega^1}d\P^1\int_{\Omega^0}\bigg(\frac{d\P^\mu}{d\P}\bigg)^2 d\P^{0,k}=\int_{\Omega^0}d\P^{0,k}\int_{\Omega^1}\bigg(\frac{d\P^\mu}{d\P}\bigg)^2 d\P^1.
\]
Recall that $d\P^\mu/d\P=\sexp(U)_T$, with $U_T=\int_0^T\sigma^{-1}\tilde{b}(s,X_s,\mu,\alpha_s)dW_s$ for some admissible control $\alpha $. Observe that $U_T$ depends only on the Brownian motion $W=W^1$ so that, for any $\omega^0\in\Omega^0$, it holds $\E^{\P^1}[(\sexp(U)_T)^2]\le \exp(c^2T)=M$ where $M$ is the constant in \eqref{eq:compact_measures}. Combining with the above inequality, we conclude that
\begin{equation}\label{eq:squarebound}
\int_{\Omega^1}(\hat{Z}_k)^2 d\P^1\le\int_{\Omega^0}d\P^{0,k}\int_{\Omega^1}\bigg(\frac{d\P^\mu}{d\P}\bigg)^2 d\P^1\le M.
\end{equation}
We have shown that the map
\begin{align}\label{eq:defPhi}
\Phi:\fxspace^{|\cV|}&\to \fxspace^{|\cV|}\\
(\fm_1,\ldots,\fm_{|\cV|})&\mapsto (\fm'_1,\ldots,\fm'_{|\cV|})\nonumber
\end{align}
is well defined and we now seek for a fixed point by means of Schauder's Theorem. Since $\fxspace^{|\cV|}$ is compact, it is enough to show that $\Phi$ is continuous.
\begin{proposition}
	$\Phi:\fxspace^{|\cV|}\to\fxspace^{|\cV|}$ is continuous.
\end{proposition}
\begin{proof}
	Consider a sequence $\{\fm^n_k\}_{n\in\N} \subset \fxspace$ converging to some $\fm^\infty_k\in\fxspace$ in the $\tau$-topology, for any $k=1,\ldots,|\cV|$. 
	For any $n\in\N\cup\{\infty\}$, $\fM^n:=\sum_{k=1}^{|\cV|}\fm^n_k1_{A_k}$ yields an input $(\mu^n_t)_{t\in[0,T]}$ to the FBSDE \eqref{eq:randomFBSDE} as described in equation \eqref{eq:induced_flow}. Observe that $t \mapsto \mu^n _t (\omega^0)$ is continuous for all $\omega^0 \in \Omega^0$ by the same argument as above. From Proposition \ref{prop:exFBSDE}, there exists a solution $(X^n,Y^n,Z^n,Z^{0,n})$ (recall $M^n\equiv 0$ from the measurability of $\fM^n$) on a probability space $(\Omega,\cF,\F,\P^n)$, where $\P^n\sim \P$ is constructed via a suitable Girsanov transformation.
	By construction of $\P^n$ for $n\in\N\cup\{\infty\}$, we have $\frac{d\P^n}{d\P^\infty}=U^n_T$, where
	\[
	U^n_t:=\exp\left(\int_0^t\sigma^{-1}\Delta^n \tilde{b}_s dW_s-\frac{1}{2}\int_0^t |\sigma^{-1}\Delta^n \tilde{b}_s|^2ds\right)
	\]
	and $\Delta^n \tilde{b}_t:=b(t,X_t,\langle h,\mu^n_t\rangle,\alpha^n_t)-b(t,X_t,\langle h,\mu^\infty_t\rangle,\alpha^\infty_t)+k(0)\langle 1,\mu^n_t-\mu^\infty_t\rangle+\int_0^tk'(t-s)\langle 1, \mu^n_s-\mu^\infty_s\rangle ds$.
	We claim that $\Delta^n \tilde{b}\to 0$ in $\lambda\otimes \P^\infty$-measure, where $\lambda$ is the Lebesgue measure on $[0,T]$. Since $\sigma^{-1}\tilde{b}$ is bounded, we deduce
	\[
	\cH(\P^{\infty}\mid\P^n)=-\E^\infty[\log(d\P^n/d\P^\infty)]=\frac{1}{2}\E^\infty\left[\int_0^{T}|\sigma^{-1}\Delta^n \tilde{b}_t|^2 dt\right]\to 0,\quad n\to \infty,
	\]
	where $\cH$ denotes the relative entropy.
	By Pinsker's inequality, $\P^n$ converges to $\P^\infty$ in total variation and, in particular, the same is true for $\fm'_k(\cdot):=\P^n(A_k\times \cdot )/\P^0(A_k)$ for every $k=1,\ldots,|\cV|$. The thesis follows from the fact that total variation convergence implies convergence in the $\tau$-topology.
	
	It remains to prove the claim. 
	We denote by 
	\[
	\ell^n(t,\omega^0):=\langle h,\mu^n_t(\omega^0)\rangle=\sum_{k=1}^{|\cV|}1_{A_k}(\omega^0)\int_{\cC([0,T];\R^{d})} h(x_t)1_{\{\tau(x)>t\}}\fm^n_k(dx),\qquad \omega^0\in\Omega^0,
	\]
	recalling from Definition \ref{def:liftedMFGeq} that $\tau(x)=\inf\{t\in[0,T] : x_t\notin \cO\}$.
	Since $h$ is bounded by \ref{hp:cont_bdd} and $\fm^n_k\to\fm^\infty_k$ in the $\tau$-topology, the sequence $\ell^n$ converges pointwise to $\ell^\infty$. In particular, it also converges in  $\lambda\otimes \P^\infty$-measure. Similarly, $\langle 1, \mu^n_s-\mu^\infty_s\rangle$ converges pointwise to $0$ from which we immediately deduce that $k(0)\langle 1,\mu^n_t-\mu^\infty_t\rangle+\int_0^tk'(t-s)\langle 1, \mu^n_s-\mu^\infty_s\rangle ds\to 0$  in  $\lambda\otimes \P^\infty$-measure. This proves one part of the claim. 
	
	Next, we prove that the optimal control $\hat{\alpha}^n _t:=\hat{a}(t,X_t,\mu^n,Z^n_t)$ converges in $\lambda\otimes \P^\infty$-measure to $\hat{\alpha}_t ^\infty :=\hat{a}(t,X_t,\mu^\infty,Z^\infty_t)$.
	Recall that $\hat{a}: [0,T]\times \R^{d} \times \cC([0,T];\subp(\R^d))\times\R^{d}\to A$ as in \eqref{eq:existence max} is a continuous function so that we only need to show 
	\[\lim_{n\to\infty}\E^{\infty}\left[\int_0^{\tauex}|Z^n_t-Z^\infty_t|^2dt\right]=0,
	\]
	which, from stability results for BSDE (e.g. \cite[Theorem 2.4]{BriandHu}), is implied by 
	\begin{equation}\label{eq:stabilityBSDE}
	\E^{\infty}\left[\int_0^{\tauex}|\hat{H}(t,X_t,\mu^n,Z^{\infty}_t)-\hat{H}(t,X_t,\mu^\infty,Z^{\infty}_t)|^2dt\right]\to 0.
	\end{equation}
	Denote by $|\Delta^n \hat{H}_t| ^2$ the above integrand and recall that $\hat{H}$ is continuous by Berge's maximum theorem. Moreover, $\hat{H}$ depends on $\mu^n$ only through the scalar functions $\ell^n$ and  $\langle 1, \mu^n\rangle$. From the pointwise convergence of $\ell^n$ and $\langle 1, \mu^n\rangle$, we have $\Delta^n \hat{H}_t\to 0$ pointwise as $n\to\infty$. Moreover, assumption \ref{hp:cont_bdd} and the definition of $\hat{H}$ imply that, for some constant $c>0$, $|\Delta^n\hat{H}_t|^2 \le c(1+|Z^\infty_t|^2)$ where the latter is $\lambda\otimes \P^{\infty}$ integrable by definition of solution. By the dominated convergence theorem \eqref{eq:stabilityBSDE} follows.
	
	Since $b$ is continuous and bounded by \ref{hp:cont_bdd} and $(\ell^n,\alpha^n)$ converges in $\lambda\otimes \P^\infty$-measure to $(\ell^\infty,\hat{\alpha}^\infty)$, the claim follows. 
\end{proof}
\begin{corollary}\label{cor:fixedpointPhi}
	The map $\Phi$ defined in \eqref{eq:defPhi} admits a fixed point  $(\hat{\fm}_1,\ldots,\hat{\fm}_{|\cV|})$.
\end{corollary}
\begin{remark}\label{rmk:set-valued}
	In order to prove the existence of a fixed point for $\Phi$ we used that $\hat{a}$ as in \eqref{eq:existence max} is a continuous function. This is a consequence of Berge's Theorem and the fact that $\hat{A}(\theta):=\{a\in A : H(\theta,a)=\hat{H}(\theta)\}$ is a singleton, for each $\theta=(t,x,m,z)$, thanks to the strict concavity of $f$ from \ref{hp:terminal}. This assumption can be relaxed to $\hat{A}(\theta)$ being a concave set (not necessarily a singleton) for each $\theta$. The only difference in the proof is that the map $\Phi$ would need to depend on an extra variable representing a given control and that we would need to use a set-valued fixed point theorem as that of \cite[Proposition 7.4]{CL15}. 
\end{remark}
\subsection{Tightness results}\label{sec:tight}
We continue to work in the setting of Section \ref{sec:fixedpoint}, in particular, we work on a probability space $(\Omega,\cF,\F,\P)$ of product type, supporting a Brownian Motion $W^0$ with the role of the common noise. We consider a sequence of discretization of $W^0$ with the following properties: 
$\{V^n\}_{n\in\N}$ is a collection of processes on $\Omega^0$ of the form 
\[V^n_t(\omega^0)=\sum_{i=1}^{2^n} v_{i-1}(\omega^0)1_{[t_{i-1},t_i)}(t)+v_{2^n-1}(\omega^0)1_{\{T\}}(t),
\]
where, for every $i=0,\ldots,2^n-1$, $t_i=iT/{2^n}$ and $v_i$ is $\cF^0_{t_i}$-measurable and takes values in a prescribed finite set (dependent on $n$). Moreover, we also require that on the event $\{\sup_{0\le t\le T} |W^0_t|\le 4^n-1\}$, $V^n$ satisfies
\begin{equation}\label{eq:discret_points}
\left|V_{t_i}^n-W^0_{t_i}\right|\le \frac{1}{2^n},\quad \forall n\in\N.
\end{equation}
We can easily construct a process with the above properties with the help of the projection map $\Pi^{(n)}:\R^d\to\R^d$ defined componentwise, for $i=1,\ldots d$, by $\Pi^{(n)}_{i}(x)=4^{-n}\lfloor4^nx_i\rfloor$ for $|x_i|\le 4^n$ and $\Pi^{(n)}_i(x)=4^n \sign(x_i)$ for $|x_i|> 4^n$. Indeed, we can set $v_0=0$ and proceed iteratively by projecting the increments of $W^0$, namely, $v_i=v_{i-1}+\Pi^{(n)}(W^0_{t_i}-W^0_{t_{i-1}})$ for $i=1,\ldots,2^n-1$.

Let $\cV^n:=\{A^n_1,\ldots,A^n_{|\cV^n|}\}$ be the finite $\sigma$-algebra generated by $V^n$ and note that $\P^0(A^n_k)>0$ for every $k=1,\ldots,|\cV^n|$. Let $(\hat{\fm}^n_1,\ldots,\hat{\fm}^n_{|\cV^n|})$ be a fixed point of $\Phi$ from Corollary \ref{cor:fixedpointPhi}.
We define 
\begin{equation}\label{eq:fMn}
	\fM^n(\omega^0):=\sum_{k=1}^{|\cV^n|}\hat{\fm}^n_k1_{A^n_k}(\omega^0),\qquad \omega^0\in\Omega^0.
\end{equation}  By constructing $\mu^n$ as in \eqref{eq:induced_flow}, Proposition \ref{prop:exFBSDE} guarantees the existence of an optimal control $\hat{\alpha}^n$ for \eqref{eq:MFG_state}, \eqref{eq:MFG_cost} under the measure $\P^{\mu^n,\hat{\alpha}^n}\in\prob(\Omega)$.
The fact that $(\hat{\fm}^n_1,\ldots,\hat{\fm}^n_{|\cV^n|})$ is a fixed point guarantees that
  \begin{equation}\label{eq:mnkfixpoint}
  \hat{\fm}^n_k(\cdot)=\P^{\mu^n,\hat{\alpha}^n}((X,W)^{-1}(\cdot)\mid A^n_k),\qquad  \forall k=1,\ldots, |\cV_n|.
  \end{equation}
We then set  
\begin{equation}\label{eq:fAn}
	\fA^n(\omega^0):=\law\left(\delta_{\hat{\alpha}_t(\omega^0,\cdot)} dt\right),\qquad \omega^0\in\Omega^0.
\end{equation} 
Note that $\env^n:=(\fM^n,\fA^n)$ is a random variable with values in $\prob(\Omega)\times \prob(\cM([0,T]\times A))$. 
\begin{notation}\label{notation:Pmu}
	To ease the notation we will simply denote $\P^n:=\P^{\mu^n,\hat{\alpha}^n}$ the measure associated to the input $\mu^n$ and the optimal control $\hat{\alpha}^n$ obtained from Proposition \ref{prop:exFBSDE}.
\end{notation}
\begin{lemma}\label{lem:tight}
	The sequence $\P^n\circ(X,W,W^0,\env^n)^{-1}$ is tight in the space of probability measures on the canonical space $\cC([0,T];\R^{3d})\times\prob(\cC([0,T];\R^{2d}))\times\prob(\cM([0,T]\times A))$. 
\end{lemma}
\begin{proof}
	We only need to prove the tightness of the marginals. For $\P^n\circ(W,W^0)^{-1}$ there is nothing to show since $W,W^0$ are independent Brownian motions under $\P^n$, for all $n \in \N$. 
		We prove the tightness of $\{\P^n\circ X^{-1}\}_{n\in\N}$ by means of Aldous' criterion.
		Recall that, the dynamics of $X$ with respect to $\P^n$ is given by
		\begin{equation}\label{eq:dynamicsn}
			dX_t=b(t,X_t,\langle h,\mu^n_t\rangle,\hat{\alpha}^n_t)dt+\sigma dW_t+\sigma^0 dW^0_t +\eta(t)d\smL^n_t.
		\end{equation}
		Moreover, the process $N^n:=(N^n_t)_{t\in[0,T]}$ defined by $N^n_t(\omega):=\int_0^t\eta(s)d\smL^n_s(\omega)$, where $\smL^n_s(\omega):=\int_0^sk(s-u)\langle 1,\mu^n_u\rangle(\omega) du$ for any $s\in[0,T]$,
		is uniformly bounded on $[0,T]$ since, for any $\F^n$-stopping times $0\le\rho\le \tau\le T$ $\P^n$-a.s.,
		\begin{equation}\label{eq:boundjumps}
		N^n_\tau-N^n_\rho\le \max_{0\le t\le T}\eta(t)\big(\smL^n_\tau-\smL^n_\rho\big)\le \max_{0\le t \le T}\eta(t)\max_{0\le t\le T}k(t)\int_\rho^\tau\langle 1,\mu^n_t\rangle dt\le C(\tau-\rho),
		\end{equation}
		for some $C>0$ dependent only on assumption \ref{hp:cont_bdd} and Definition \ref{def:kernel}. 
		Together with \ref{hp:cont_bdd}, it clearly implies
		\[
		\sup_{0\le t\le T}\E^n|X_t|^2\le C,
		\]
		for some possibly larger constant that we still denote by $C$. Using Markov's inequality, we deduce that for any $\varepsilon>0$ there exists $K>0$ such that, for any $n\in\N$,		
			\[\P^n\left(\sup_{0\le t\le T}|X_t|>K\right)\le \varepsilon,
			\]
			which is the first requirement for applying \cite[Theorem VI.4.5]{bookJacod}.
			The second requirement is obtained using again Markov's inequality and the following estimation. 
			Let $\rho\le \tau$, $\F^n$-stopping times with $\tau-\rho\le 1$.
	\begin{eqnarray*}
		\E^n |X_{\tau}-X_{\rho}|
		&\le& \E^n\left[\int_\rho^{\tau}|b(s,X_s,\langle h,\mu^n_s\rangle,\hat{\alpha}^n_s)|ds+\left|\int_\rho^{\tau}\eta_sd\smL^n_s\right|\right]\\
		&&+\E^n\left|\int_\rho^{\tau}\sigma dW_s+\int_\rho^{\tau}\sigma^0 dW^0_s\right|,\\
		&\le & C \sqrt{\tau-\rho}
	\end{eqnarray*}
	where we used \ref{hp:cont_bdd} and \eqref{eq:boundjumps}.
	By Aldous' criterion $\{\P^n \circ X^{-1}\}_{n\in\N}$ is tight. 
	
	As $A$ is compact, $\{\P^n\circ (\fA^n)^{-1}\}_{n\in \N}$ is clearly tight. Thus, it remains to show that $\{\P^n\circ (\fM^n)^{-1}\}_{n \in \N}$ is tight. The argument is exactly the one used in the proof of \cite[(Vol II) Lemma 3.16]{bookMFG}, we provide it for completeness. Fix $\varepsilon>0$. From the first part of the proof $\P^n\circ (X,W)^{-1}$ is tight so that for an arbitrary $j\in\N$ we can find a compact set $K_j\subset \cC([0,T];\R^{2d})$ such that $\P^n((X,W)\notin K_j)\le \varepsilon /4^j$ for every $n\in\N$. Since $\fM^n$ is a version of the conditional distribution of $(X,W)$ given $\sigma(\cV^n)$, we have
	\[ \sup_{n\in \N} \E^n\big[\fM^n\big(K^C_j\big)\big]=\sup_{n\in \N} \P^n((X,W)\notin K_j)\le \frac{\varepsilon}{4^j}.
	\]
	The set $\cK:=\{\Q\in\prob(\cC[0,T];\R^{2d})\ : \ \Q(K^C_j)\le 2^{-j}\ \ \forall j\in\N\}$ is relatively compact. Moreover, we can use the above estimation and Markov's inequality to deduce
	\[
	\sup_{n\in \N} \P^n(\fM^n\notin \cK)\le\sup_{n\in \N}\sum_{j\in \N} \P^n\big(\fM^n\big(K^C_j\big)> 2^{-j}\big)\le 2\varepsilon.
	\]
	Since $\varepsilon>0$ is arbitrary, the thesis follows. 
\end{proof}
Lemma \ref{lem:tight} implies that the sequence of probability measures $\P^n\circ(X,W,W^0,\env^n)^{-1}$, whose canonical space is $\cC([0,T];\R^{2d})\times \cC([0,T];\R^{d})\times \prob(\cC([0,T];\R^{2d}))\times \prob(\cM([0,T]\times A))$, admits a convergent subsequence. We denote by $\P^\infty\circ(X,W,W^0,\env^\infty)^{-1}$ its weak limit and we show next some of its crucial properties.
Since we are considering convergence in distribution, the limiting optimal control will take the form of a relaxed control. Indeed, equation \eqref{eq:fAn} reveals that $\fA^n$ is the conditional law of the random measure $\delta_{\hat{\alpha}_t^n} dt$. To recover the (unconditional) law of the process one simply take the integral $\int_{\Omega^0}\fA^n d\P^n$.  
\begin{lemma}\label{lem:limit_cond}
Denote by $\P^\infty\circ(X,W,W^0,\env^\infty)^{-1}$ the weak limit of a convergent subsequence of $\P^n\circ(X,W,W^0,\env^n)^{-1}$.
	Then, $\fM^\infty$ is a version of the conditional distribution of $(X,W^\infty)$ given $(W^0,\env^\infty)$. Similarly, $\fA^\infty$ is a version of the conditional distribution of $\int_{\Omega^0}\fA^\infty d\P^\infty$ given $(W^0,\env^\infty)$.
\end{lemma}
\begin{proof} We only prove the first statement, the other one is analogous. In order to do so we follow closely the proof of \cite[(Vol II) Lemma 3.28]{bookMFG}. To prove the thesis we can equivalently show that for every bounded and uniformly continuous functions $h^0:\cC([0,T];\R^{d})\times\prob(\cC([0,T];\R^{2d}))\times \prob(\cM([0,T]\times A)) \to \R$ and $h^1:\cC([0,T];\R^{d})\times \cC([0,T];\R^{d})\to \R$, we have
	\[
	\E^{\infty}[h^0(W^0,\env^\infty)h^1(X,W)]=\E^\infty\left[h^0(W^0,\env^\infty)\int_{\cC([0,T];\R^{2d})}h^1(x,w)d\fM^{\infty}(x,w)\right].
	\]
	By construction $\fM^n$ is the conditional distribution of $(X,W)$ given $V^n$, so that, the equality $l_n=r_n$ with $l_n:=\E^n[h^0(V^n,\env^n)h^1(X,W)]$ and $r_n:=\E^n[h^0(V^n,\env^n)\int_{\cC([0,T];\R^{2d})}h^1(x,w)d\fM^{n}(x,w)]$ holds for any $n\in\N$; we want to pass to the limit.
	Recall that $W^0$ is a Brownian motion under every $\P^n$, thus, $\lim_{n\to\infty}\P^n(B^C_n)=0$ where $B_n:=\{\sup_{0\le t\le T} |W^0_t|\le 4^n-1\}$.
	Using \eqref{eq:discret_points} on the set $B_n$, we deduce
	\[\lim_{n\to\infty}\P^n\left(\sup_{0\le t\le T}|V^n_t-W^0_t|\le \frac{1}{2^{n}}+\sup_{0\le s\le t\le T, |t-s|\le 2^{-n}}|W^0_s-W^0_t| \right)=1.
	\]
	From this fact, the weak convergence of $\P^n\circ(X,W,W^0,\env^n)^{-1}$ and the fact that $h^0$ is bounded and uniformly continuous, we deduce that
	 $l_n$ converges to $\E^{\infty}[h^0(W^0,\env^\infty)h^1(X,W)]$ as $n\to\infty$. In a similar fashion, $r_n$ converges to $\E^\infty[h^0(W^0,\env^\infty)\int_{\cC([0,T];\R^{2d})}h^1(x,w)d\fM^{\infty}(x,w)]$ as $n\to\infty$. The thesis follows from $l_n=r_n$ for every $n\in\N$.
\end{proof}
For any $n\in\N\cup\{\infty\}$, let $\mu^n=(\mu^n_t)_{t\in[0,T]}$ be the process of sub-probabilities induced by $\fM^n$ as in \eqref{eq:induced_flow}.
\begin{lemma} \label{lem:weakconv_subp}
Let $\P^\infty\circ(X,W,W^0,\fM^\infty)^{-1}$ be a weak limit of $\P^n\circ(X,W,W^0,\fM^n)^{-1}$. For any bounded Lipschitz continuous function $g:\R^d\to\R$, the process $(\langle g,\mu^n_t\rangle)_{t\in[0,T]}$ converges weakly to $(\langle g,\mu^\infty_t\rangle)_{t\in[0,T]}$.
	
\end{lemma}
\begin{proof}
	For any $n\in\N$, shortly denote with $\ell^n:\Omega^0\to \cC([0,T];\R)$ the random variable $\ell^n:=\langle g,\mu^n\rangle$.
	To prove that $\ell^n$ converges weakly to $\langle g,\mu^\infty\rangle$ we show: 
	\begin{itemize}
		\item $\P^n\circ(\ell^n)^{-1}$ is tight;
		\item for any $\kappa \in \N$, for any $t_1,\ldots,t_\kappa$, the vector $(\ell^n_{t_1},\cdots,\ell^n_{t_k})$ weakly converges to $(\langle g,\mu_{t_1}^\infty\rangle,\cdots,\langle g,\mu_{t_k}^\infty\rangle)$.
	\end{itemize}
	The desired result follows by \cite[VI 3.20]{bookJacod}.\smallskip
	
	\textit{Tightness}. We show that $\ell^n$ takes values in a relatively compact subset of $\cC([0,T];\R)$.
	First, the assumption that $g$ is bounded implies that $\ell^n(\omega^0)$ is uniformly bounded in $n$ and $\omega^0$.
	Second, we show uniform equicontinuity. Take an arbitrary $n$ and $\omega^0$. By construction, there exists $\hat{\fm}^n_k\ll\P^1$ such that $\fM^n(\omega^0)=\hat{\fm}^n_k$. Moreover, recall that \eqref{eq:squarebound} yields $\frac{d\hat{\fm}_k}{d\P^1}=\hat{Z}^n_k$ with $\E^{\P^1}[(\hat{Z}^n_k)^2]\le M$ for some $M>0$ depending only on the coefficients of the problem. For any $0\le s\le t\le T$,
	\begin{eqnarray*}
		|\ell^n(t,\omega^0)-\ell^n(s,\omega^0)|&\le&\int_{\cC([0,T];\R^{d})}|g(x_t)1_{\{\tau(x)>t\}}-g(x_s)1_{\{\tau(x)>s\}}|\hat{Z}^n_k(x) \P^1(dx)\\
		&\le& \left(\E^{\P^1}\left[|g(X_t)1_{\{\tauex>t\}}-g(X_s)1_{\{\tauex>s\}}|^2\right]\right)^{\frac{1}{2}}\left(\E^{\P^1}[(\hat{Z}^n_k)^2]\right)^{\frac{1}{2}}.
	\end{eqnarray*}
	The last term on the r.h.s.\ is uniformly bounded by $\sqrt{M}$ so we will focus on the other one.
	As $g$ is bounded by some constant $c>0$, it holds
	\[
	\int_{\{s<\tau(x)\le t\}} |g(x_s)|\P^1(dx)\le c\P^1(\{s<\tau(x)\le t\})=c(F_{\tauex}(t)-F_{\tauex}(s)),
	\]
	where $F_{\tauex}$ is the distribution function of $\tauex$ under $\P^1$. The latter is continuous since $\P^1(\{\tau(x)=u\})=0$ for every $u\in[0,T]$ (see, e.g., \cite[Lemma A.4]{CGL19}). In particular, $F_{\tauex}$ is uniformly continuous on $[0,T]$ so that $F_{\tauex}(t)-F_{\tauex}(s)$ can be made arbitrarily small provided that $|t-s|$ is sufficiently small. Second, from the Lipschitz continuity of $g$ we deduce that there exists a constant $c>0$ such that
	\[
	\int_{\{\tau(x)> t\}} |g(x_t)-g(x_s)|^2\P^1(dx)\le c \E^{\P^1}[|X_t-X_s|^2]= c (t-s)\left(\sum_{i,j=1}^d|\sigma_{i,j}|^2+|\sigma^0_{i,j}|^2\right),
	\]
	where the bound is again uniform. Putting together the two inequalities we obtain the uniform equicontinuity of $\ell^n(\omega^0)$.
	By the Ascoli-Arzel\`a Theorem we conclude that $\{\ell^n\}_{n\in\N}$ takes values in a relatively compact set. \smallskip
	
	\textit{Convergence of the finite dimensional distributions.}	
	Towards this goal we first show that $\P^\infty\circ X^{-1}\ll\P^1$ using an argument similar to \cite[Proposition 3.1]{CGL19}.
	As in \eqref{eq:squarebound}, the sequence $d\P^{\mu^n}/d\P$ can be shown to be uniformly $L^2$-bounded by a constant $M$. Thus, for any sequence of $B_n\in\cF^1$ such that $\P^1(B_n)\to 0$ we have  
	\[\P^n\circ X^{-1}(B_n)=\E\left[\frac{d\P^{\mu^n}}{d\P}1_{B_n}\right]\le M\P^1(B_n)\to 0.
	\]
	Let $U^n _t=\int_0^t\sigma^{-1}\tilde{b}(s,X_s,\mu,\hat \alpha^n _s)dW_s$, $t\in [0,T]$. Next, we prove that the sequence of martingales $\sexp(U^n)=(\sexp(U^n)_t)_{t\in[0,T]}$ such that $d\P^{\mu^n}/d\P=\sexp(U^n)_T$ is tight. Using \cite[Theorem VI.4.13]{bookJacod}, it is enough to check the tightness of the sequence of their quadratic variations. There exists a constant $C>0$, depending only on the coefficients of the problem, such that for any couple of stopping times $\sigma\le \tau$ with $\tau-\sigma\le \delta$, it holds
	\begin{eqnarray*}
	\P(|\langle \sexp(U^n)\rangle_\tau-\langle \sexp(U^n)\rangle_\sigma |\ge \varepsilon)&\le&\varepsilon^{-1}\E\left[|\langle \sexp(U^n)\rangle_\tau-\langle \sexp(U^n)\rangle_\sigma |\right]\\
	&\le&\varepsilon^{-1}\E\left[\int_{\sigma}^{\tau}|\sigma^{-1}\tilde{b}(s,X_s,\mu^n,\hat{\alpha}^n_s)\sexp(U^n)_s|^2\right]\\
	&\le& C\delta,
	\end{eqnarray*}
 where we used in the last inequality \ref{hp:cont_bdd} and the fact that $\sexp(U^n)_s$ is also uniformly $L^2$-bounded. By Aldous' criterion, $\langle \sexp(U^n)\rangle$ is tight.
 Using \cite[Theorem X.3.3]{bookJacod}, these two properties imply $\P^\infty\circ X^{-1}\ll\P^1$. 
 
 The absolute continuity property implies that, for any $t\in[0,T]$, the set $\{\tauex=t\}$ has probability zero under $\P^\infty$. 
Since, by Lemma \ref{lem:limit_cond}, $\fM^\infty$ is a version of the conditional distribution of $(X,W^\infty)$ given $(W^0,\fM^\infty)$, it holds
 \[0=\P^{\infty}(\{\tauex=t\})=\int_{\Omega^0}\fM^\infty(\omega^0,\{\tauex=t\})\P^0(d\omega^0).
 \]
 We deduce that $\fM^\infty(\{\tauex=t_1\}\cup\cdots\cup \{\tauex=t_\kappa\})=0$ $\P^0$-a.s. On the space of $\Q\in\prob(\cX)$ satisfying $\Q(\{\tauex=t_j\})=0$ for every $j=1,\ldots,\kappa$, the map
 \[\iota:\Q\mapsto\left(\int_\cX g(x_{t_1}) 1_{\{\tau(x)>t_1\}}\Q(dx),\ldots, \int_\cX g(x_{t_\kappa}) 1_{\{\tau(x)>t_\kappa\}}\Q(dx)\right)
 \]
 is continuous with respect to the weak topology by the continuous mapping theorem. Since $\fM^n \to \fM^\infty$ in distribution by Lemma \ref{lem:limit_cond}, the continuity of $\iota$ ensures that also $\iota(\fM^n)\to\iota(\fM^\infty)$ in distribution. As $\iota(\fM^n)=\langle g,\mu^n\rangle$ for any $n\in\N\cup\{\infty\}$ the desired convergence is proved. 
\end{proof}

\subsection{Proof of Theorem \ref{thm:existence}}
We have now all the necessary ingredients to prove the main existence result. 
In order to perform the limiting argument, we consider first the relaxed version of the control problem, namely, where the control policies are random variables with values in the compact set 
\[
\Gamma:=\left\{q\in\cM([0,T]\times A)\ :\ q(dt,da)=dt\ q_t(da),\text{ for some Borel measurable kernel }q_t\right\}.
\]
For any $n\in\N$, recall that $\hat{\alpha}^n$ denotes the optimal control of the approximated problem.
We rewrite \eqref{eq:dynamicsn} with $\hat{\gamma}^n_s(da)=\delta_{\hat{\alpha}^n_t}(da)$ as
\[
X_t=X_0 +\int_0^t\int_A b(s,X_s,\langle h,\mu^n_s\rangle,a)\hat{\gamma}^n_s(da)ds+\sigma W^\alpha_t+\sigma^0 W^0_t +\int_0^t\eta(s)d\smL^n_s,
\] with $X_0\sim\xi$. Assumptions \ref{hp:affine} and \ref{hp:terminal} guarantee that the optimization over relaxed controls, namely, the maximization of
\begin{equation}\label{eq:MFG_cost_relax_n}
	J^n(\gamma)=\E\bigg[\int_0^{\tauex} \int_A f(t,X_t,\langle h,\mu^n_t\rangle,a)\gamma_t(da)dt+G(\tauex,\Xtauex)\bigg],
\end{equation}
over the class of $\F$-progressively measurable processes $(\gamma_t)_{t\in [0,T]}$ such that $\gamma_t (da)dt$ takes values in $\Gamma$, yields the same values as the original one (see, e.g.,  \cite[Section 4.1]{CDL16} and \cite[Theorem 4.11]{elkaroui1987}).
For any of the approximated problems the $W^0$-conditional distributions of $(X,W)$ and of the optimal control $\hat{\alpha}^n$ are measurable with respect to the filtration generated by the common noise $W^0$. However, this is no longer true for the limit $\env^\infty$. The component $\fA^\infty$ induces a relaxed control $\hat{\gamma}^\infty:=\int_{\Omega^0}\fA^\infty d\P^\infty$ (see also Lemma \ref{lem:limit_cond} and the discussion before) taking values in $\Gamma$. 

We first check the admissibility of the setup as in Definition \ref{def:admissiblesetup} and second we check the requirements of Definition \ref{def:liftedMFGeq}.

\textit{Admissibility of the setup.} We choose $\Omega^0=\cC([0,T];\R^{2d})$ and $\Omega^1=\cC([0,T];\R^{d})\times \prob(\cC([0,T];\R^{2d}))\times\prob(\cM([0,T]\times A))$ so that the product space $\Omega^0\times\Omega^1$ is the canonical space to support the weak limit  $\P^\infty\circ(X,W,W^0,\env^\infty)^{-1}$ from Lemma \ref{lem:limit_cond}. We choose as filtration $\F$ the one generated by $(X,W,W^0,\env^\infty)$, in the sense of Definition \ref{def:natural}.

We check the immersion property. Let $\Theta_1:=(W^0,\env^\infty)$, $\Theta_2:=(X,W,W^0,\env^\infty)$ and $\F^{\Theta_i}$ the filtration generated by $\Theta_i$, for $i=1,2$. As $\F^{\Theta_1}$ is automatically immersed in the natural filtration of $(X_0,W,W^0,\env^\infty)$, it is enough to check that $\F^{\Theta_1}$ is immersed in $\F^{\Theta_2}$. A sufficient condition for the immersion property is that $\cF^{\Theta_1}_T$ and $\cF^{\Theta_2}_t$ are conditionally independent given $\cF^{\Theta_1}_t$ for every $t\in[0,T]$ (see e.g. \cite[(Vol II) Lemma 1.7]{bookMFG}). 
Fix $t\in [0,T]$ and let $A\in\cF^{\Theta_1}_T$ and $B\in\cF^{\Theta_2}_t$. We need to show
\begin{equation} \P(A\cap B\mid\cF^{\Theta_1}_t)=\P(A\mid\cF^{\Theta_1}_t)\P(B\mid\cF^{\Theta_1}_t). \label{eq:cond-prob}
\end{equation}
Let 
$e_t(x,w,q)=(x_{\cdot\wedge t},w_{\cdot\wedge t},q_{|_{[0,t]\times A}})$ for $(x,w,q)\in \cC([0,T];\R^{2d})\times \Gamma$. With $\env^\infty_{\cdot\wedge t}$ we mean $\env^\infty\circ e_t^{-1}$.
We can take w.l.o.g. 
\[ A=\{W^0 \in A_1\} \times \{\env^\infty \in A_2\}, \quad B= \{X_{\cdot \wedge t} \in B_1\} \times \{W_{\cdot \wedge t} \in B_2\} \times \{W^0 _{\cdot \wedge t} \in B_3\} \times \{\env^\infty_{\cdot \wedge t} \in B_4\},\] 
for some $A_1,B_1,B_2,B_3$ and $A_2,B_4$ Borel sets of $\cC([0,T];\R^{d})$ and $\prob(\cC([0,T];\R^{2d}))$ respectively, so that
\[\P(A\cap B\mid\cF^{\Theta_1}_t)=\E[1_{A_1}(W^0)1_{A_2}(\env^\infty)1_{B_1}(X_{\cdot\wedge t})1_{B_2}(W_{\cdot\wedge t})1_{B_3}(W^0_{\cdot\wedge t})1_{B_4}(\env^\infty_{\cdot\wedge t})\mid\cF^{\Theta_1}_t].
\]
Since $(X,W)$ is defined on $\Omega^1$ and $(W^0,\env^{\infty})$ on $\Omega^0$ we can replace, in the above equation, $1_{B_1}(X_{\cdot\wedge t})1_{B_2}(W_{\cdot\wedge t})$ with its conditional expectation with respect to $\cF^{\Theta_1}_t$. Since the latter is equal to  $\fM^\infty_{\cdot\wedge t}(B_1\times B_2)$ by construction of $\env^\infty$, we can use the measurability property of $\fM^\infty_{\cdot\wedge t}$ to deduce
\begin{eqnarray*}
	\P(A\cap B\mid\cF^{\Theta_1}_t)&=&\fM^\infty_{\cdot\wedge t}(B_1\times B_2)1_{B_3}(W^0_{\cdot\wedge t})1_{B_4}(\env^\infty_{\cdot\wedge t})\E[1_{A_1}(W^0)1_{A_2}(\env^\infty)\mid\cF^{\Theta_1}_t]\\
	&=&\P(B\mid\cF^{\Theta_1}_t)\P(A\mid\cF^{\Theta_1}_t),
\end{eqnarray*}
as desired.

\textit{The limiting $\env^\infty$ induces a relaxed weak MFG equilibrium.} 
Consider the relaxed control $\hat{\gamma}^\infty$, induced by $\env^\infty$, which is the weak limit of the sequence of relaxed controls $\hat{\gamma}^n$ that are optimal for the approximated problems. Using the assumptions \ref{hp:cont_bdd} and \ref{hp:affine}, we observe that the function
\[
(x,m,\gamma)\mapsto \int_0^\cdot \int_A \tilde{b}(s,x_s,m,a)\gamma_s(da)ds,
\]
with $\tilde{b}$ as in \eqref{eq:btilde}, is continuous and bounded from $\cC([0,T];\R^d)\times \cC([0,T];\R)\times \Gamma $ to $\cC([0,T];\R^d)$. 
 Lemma \ref{lem:weakconv_subp} and the fact that $\hat{\gamma}^\infty$ is the weak limit of $\hat{\gamma}^n$ guarantee that $X$ satisfies 
\begin{equation}\label{eq:limiting_SDE_eq}
X_t=X_0+\int_0^t\int_A b(s,X_s,\langle h,\mu^\infty_s\rangle,a)\hat{\gamma}^\infty_s(da)ds+\sigma W_t+\sigma^0 W^0_t +\int_0^t\eta(s)d\smL^\infty_s,
\end{equation}
 on $(\Omega,\cF^\infty,\F^\infty,\P^\infty)$ with $X_0\sim\xi$. We next show that $J(\hat{\gamma}^\infty)=\lim_{n\to\infty}J^n(\hat{\gamma}^n)$, with $J^n$ as in \eqref{eq:MFG_cost_relax_n}. Under $\P^{\infty}$, $X$ has a bounded drift thanks to  \eqref{eq:limiting_SDE_eq} and \ref{hp:cont_bdd}. It follows, see, e.g., \cite[Lemma D.3]{CF18}, that the functions $x\mapsto \tau(x)$ and $x\mapsto1_{[0,\tau(x))}$, with $\tau(x)$ as in Definition \ref{def:MFGequilibrium_weak}, are $\P^{\infty}$-a.s.\ continuous and bounded. The weak convergence of $\hat{\gamma}^n$ to $\hat{\gamma}^\infty$ and that of $\langle h,\mu^n_t\rangle$ to $\langle h,\mu^\infty_t\rangle$ from Lemma \ref{lem:weakconv_subp} yield $J(\hat{\gamma}^\infty)=\lim_{n\to\infty}J^n(\hat{\gamma}^n)$.
 Note that exactly the same argument implies that $J(\gamma)=\lim_{n\to\infty}J^n(\gamma)$ where $\gamma=(\gamma_t)_{t\in [0,T]}$ is an arbitrary relaxed control such that $\gamma_t (da)dt$ takes values in $\Gamma$ (i.e.\ when only $\mu^n$ varies with $n$ and $\gamma$ is fixed).
 By the optimality of $\hat{\gamma}^n$ for the approximated problem, we have $J^n(\hat{\gamma}^n)\ge J^n(\gamma)$ and, by passing to the limits, we obtain $J(\hat{\gamma}^\infty)\ge J(\gamma)$. As $\gamma$ was arbitrary, this proves that $\hat{\gamma}^\infty$ is optimal.
 
Finally, by Lemma \ref{lem:limit_cond},
\[\P^{\infty}\big(\omega^0\in\Omega^0\ :\ \fM^{\infty}(\omega^0) =\law(X(\omega^0,\cdot),W(\cdot)),\ \fA^{\infty}(\omega^0)=\law(\hat{\gamma}^\infty(\omega^0,\cdot)\big)=1
\]
so that \eqref{eq:equilibrium} in Definition \ref{def:liftedMFGeq} is satisfied (with relaxed controls).
We conclude by noting that assumptions \ref{hp:affine} and \ref{hp:terminal} guarantee that the optimization over relaxed controls yields the same values as the original one without affecting admissibility and the fixed point condition (see, e.g., \cite[Section 4.1]{CDL16} and \cite[Theorem 4.11]{elkaroui1987}).

\section{Approximate Nash Equilibria}\label{sec:approx-nash}

In this section we show how solving the mean field limit problem gives rise to approximate Nash equilibria (see Definition \ref{def:eps_Nash}) for the $N$-player game \eqref{eq:model2}-\eqref{eq:model2_cost}. 
We briefly describe here the procedure we intend to undertake to obtain existence of $\varepsilon$-Nash equilibria for games with sufficiently many players. In our analysis of the limiting mean field problem \eqref{eq:MFG_state}-\eqref{eq:MFG_cost}, we introduced a sequence of intermediary problems where the conditional distribution of $(X,W)$ is a random variable taking only finitely many values; see Section \ref{sec:fixedpoint}. The equilibrium induced by $\env^n$ yields an optimal control for the optimization problem that we denoted by $\hat{\gamma}^n$, in its relaxed form. Our first next step is to deepen this analysis and to establish the existence of an optimal control depending on the state $X$ and the Brownian motion $W$ only, providing the same value of $\hat{\gamma}^n$. We next use this control to construct $\varepsilon$-Nash equilibria for an auxiliary $N$-player game which can be seen as the pre-limit of the intermediate mean field problem. Finally, we need to show that for $n$ large enough such a control induces an approximate Nash equilibria in the $N$-player game \eqref{eq:model2}-\eqref{eq:model2_cost}.\medskip

Consider again the setup of Section \ref{sec:fixedpoint} and the equilibrium $\P^n\circ(X,W,W^0,\env^n)^{-1}$ constructed at the beginning of Section \ref{sec:tight} (see also Notation \ref{notation:Pmu}) and $\mu^n=(\mu^n_t)_{t\in[0,T]}$ the induced process of sub-probabilities as in \eqref{eq:induced_flow}. We denoted $\hat{\gamma}^n_s(da) = \delta_{\hat{\alpha}^n_s}(da)$ the (relaxed) optimal control for the dynamics
\begin{equation}\label{eq:controlledSDErelax}
	X_t=X_0+\int_0^t\int_A b(s,X_s,\langle h,\mu^n_s\rangle,a)\hat{\gamma}^n_s(da)ds+\sigma W_t+\sigma^0 W^0_t +\int_0^t\eta(s)d\smL^n_s.
\end{equation}
with $X_0 \sim \xi$ and value function $J^n(\hat{\gamma}^n)$ as in \eqref{eq:MFG_cost_relax_n}.
\begin{proposition}\label{prop:feedback}
	Let $n\in\N$. Under Assumption \ref{ass:alltogether},  there exists $g^n:[0,T]\times \Omega^1\to \prob(A)$ such that the relaxed control $\gamma^n_{s,X,W}(da):=g^n(s,X,W)(da)$ satisfies $J^n(\gamma^n_{\cdot, X,W})=J^n(\hat{\gamma}^n)$.
\end{proposition}
\begin{proof}
	We proceed by means of a suitable mimicking theorem in the spirit of \cite{CGL19,Lack15}. A crucial difference here is that we need to condition on both $(X,W)$ which is in line with the construction of the equilibrium conditional distribution $\env^n$ of Section \ref{sec:fixedpoint}. 
	Recall Notation \ref{notation:Pmu} for the measure $\P^n$  and the corresponding expectation $\E^n$.
	
	Consider the probability measure $p^n\in\prob([0,T]\times \Omega^1\times A)$ defined by
	\[p^n(B) :=\frac{1}{T}\E^n\left[\int_{[0,T]\times A}1_B(t,X,W,a)\hat{\gamma}^n_t(da)dt\right].
	\]
	We denote by $g^n$ the disintegration kernel of $p^n$ so that $p^n(dt,d\varphi,da)=g^n(t,\varphi)(da)\mathfrak{p}^n(dt,d\varphi)$ for some $\mathfrak{p}^n\in\prob([0,T]\times\Omega_1)$. We shortly denote $\hat{X}:=(X,W)$ and $\gamma^n_{t,\hat{X}}=g^n(t,\hat{X})$.
	We easily prove that $\gamma^n_{t,\hat{X}}=\E^n[\hat{\gamma}^n_t\mid\cF_t^{\hat{X}}]$ $\lambda\otimes \P^n$-a.e., where $(\cF_t^{\hat{X}})_{t\in[0,T]}$ is the natural filtration of $\hat{X}$ and $\lambda$ the Lebesgue measure on $[0,T]$. Indeed, by taking arbitrary $A'\subset A$ measurable and $B_t\in\cF_t^{\hat{X}}$ for each $t$, we have
	\begin{align*}
		\frac{1}{T}\int_0^T\E^n\left[1_{B_t}(\hat{X})\gamma^n_{t,\hat{X}}(A')\right]dt &=\int_{[0,T]\times\Omega^1}d\mathfrak{p}^n\int_{A'} 1_{B_t}(\varphi)g^n(t,\varphi)(da)\\
		&=\frac{1}{T}\E^n\left[\int_{[0,T]\times A'}1_{B_t}(\hat{X})\hat{\gamma}^n_t(da)dt\right]\\
		&=\frac{1}{T}\int_0^T\E^n\left[1_{B_t}(\hat{X})\hat{\gamma}^n_t(A')\right]dt.
	\end{align*}
	We now make use of the mimicking theorem of \cite[Corollary 3.11]{BruShre} to obtain a weak solution of the controlled SDE: 
	\begin{equation*}
		\tilde{X}_t=\tilde{X}_0+\int_0^t\int_A b(s,\tilde{X}_s,\langle h,\tilde{\mu}^n_s\rangle,a)\gamma^n_{s,\tilde{X},\tilde{W}}(da)ds+\sigma \tilde{W}_t+\sigma^0 \tilde{W}^0_t +\int_0^t\eta(s)d\tilde{\smL}^n_s,
	\end{equation*}
	with $\tilde X_0 \sim \xi$ and with the remarkable property that the joint law of $(\tilde{X},\tilde{W})$ coincide with that of $(X,W)$ in \eqref{eq:controlledSDErelax}. 
	A crucial aspect is now the construction of $W^0$ at the beginning of Section \ref{sec:fixedpoint} which reveals that $W^0$ is adapted to the natural filtration of $(X,W)$.
	Recall now that $\mu^n_s$ is obtained via \eqref{eq:induced_flow} using $\fM^n$ from \eqref{eq:fMn}. Since $\fM^n$ takes finitely many values depending only on the realization of $W^0$, we deduce that $\mu^n_s$ is $\cF_s^{\hat{X}}$-measurable for every $s$. Note now that $\tilde{\fM}^n$ induced by the control $\gamma^n_{t,\tilde{X},\tilde{W}}$ and defined via \eqref{eq:fMn}-\eqref{eq:mnkfixpoint} is necessarily equal to $\fM^n$ as the joint laws of $(X,W)$ and $(\tilde{X},\tilde{W})$ coincide. In particular, $\tilde{\mu}^n_s$ induced by $\tilde{\fM}^n$  satisfies $\tilde{\mu}^n_s=\mu^n_s$ for every $s$. Moreover, since $X$ is the canonical process, $\tauex=\tau(X)$ where $\tau(x)=\inf\{t\in[0,T] : x_t\notin \cO\}$ as defined in Definition \ref{def:liftedMFGeq}.
	This allows us to obtain equality of the value functions as follows:
	\begin{align*}
		J^n(\hat{\gamma}^n)&=\E^n\bigg[\int_0^{\tauex} \int_A f(t,X_t,\langle h,\mu^n_t\rangle,a)\hat{\gamma}^n_t(da)dt+G(\tauex,\Xtauex)\bigg]\\
		&=\E^n\bigg[\int_0^{\tauex} \int_A f(t,X_t,\langle h,\mu^n_t\rangle,a)\E^n\big[\hat{\gamma}^n_t(da)\mid\cF_t^{\hat{X}}\big]dt+G(\tauex,\Xtauex)\bigg]\\
		&=\E^n\bigg[\int_0^{T} \int_A 1_{[0,\tau(X)]}(t)f(t,X_t,\langle h,\mu^n_t\rangle,a)\E^n\big[\hat{\gamma}^n_t(da)\mid\cF_t^{\hat{X}}\big]dt+G(\tau(X),X_{\tau(X)})\bigg]\\
		&=\E^n\bigg[\int_0^{\tauex} \int_A f(t,X_t,\langle h,\mu^n_t\rangle,a)\gamma^n_{t,\hat{X}}(da)dt+G(\tauex,\Xtauex)\bigg]\\
		&=J^n(\gamma^n).
	\end{align*}
	In the second equality we simply used the tower property and the measurability of $\mu^n_t$. In the third equality we simply wrote the explicit dependence of $\tauex$ on $X$ which reveals that the term inside the expectation is a measurable function of $(X,W)$. In the fourth equality we used
	$\gamma^n_{t,\hat{X}}=\E^n [\hat{\gamma}^n_t\mid\cF_t^{\hat{X}}]$ $\lambda\otimes \P^n$-a.e.\ and equality in distributions of $(X,W)$ and $(\tilde{X},\tilde{W})$.
\end{proof}

The second step is to establish some convergence results for the empirical measure of the state dynamics of the $N$ players, if they use the control from 
Proposition \ref{prop:feedback} and if the input flow of sub-probabilities is a random variable taking only finitely many values. Recall that, at the beginning of Section \ref{sec:tight}, we considered a discretization $V^n$ of a Brownian motion which generates a partition of finitely many events $A^n_1,\ldots,A^n_{|\cV^n|}$. 

We now introduce the following approximate $N$-player game, which is intended to be constructed in a weak sense as in the beginning of section \ref{sec:MFG}. Observe that Assumption \ref{hp:sigmaconst} and equation \eqref{eq:Xi_weak_formulation} imply that $(\xi^i, W^i, W^0)_{i=1}^N$ and $(X^i, W^i)_{i=1}^N$ generate the same natural filtration. Note that the initial conditions $\xi^i$ are included in the states $X^i$. Therefore any admissible (open loop) control can be viewed as a measurable function of $(\xi^i, W^i, W^0)_{i=1}^N$ or alternatively of $(X^i, W^i)_{i=1}^N$. The latter is more convenient for the sake of exposition. Let $g=(g^1,\ldots,g^N)$ be an admissible relaxed control, so it can be represented as $g^i = g^i(t, (X^i, W^i)_{i=1}^N)$ for some suitable measurable functions denoted $g^i$, $i=1,\ldots, N$, with a little abuse of notation. The corresponding state equation is
\begin{eqnarray}\label{eq:approxSDE}
	Y^i_t&=&Y_0^i+\int_0^t\int_A \tilde{b}(s,Y^i_s, \rho^{N,n}, a)g^i _s(da)ds+\sigma W^i_t+\sigma^0 W^0_t,\qquad 0\le t\le \tauey^i\ ,\\
	Y_0^i&=&\xi^i\nonumber\ ,
\end{eqnarray}
where $\xi^i\sim\xi$ and, for  $0\le t\le T$,
\begin{equation}\label{eq:rhoN}
	\rho^{N,n}_t = \E^{N,n}\big[\mu^N_t|\cF^{\xi,W,V^n}_t\big],
\end{equation}
where $\cF^{\xi,W,V^n }_t$ denotes the minimal filtration generated by $\xi^i ,W^i, V^n$, for $i=1,\ldots, N$, and satisfying the usual conditions, and
\[\mu^N_t:=\frac{1}{N}\sum_{i=1}^N\delta_{Y^i_t}(\cdot) 1_{[0,\tauey^i)}(t),\qquad \tauey^i:=\inf\{t\in[0,T]\ :\  Y^i_t\notin \cO\}.
\]
The difference with \eqref{eq:model2} is that the players do not exactly interact with the empirical sub-distribution of players who have not been absorbed. Indeed, the latter is integrated with respect to the common noise component on a finite number of events $A^n_1,\ldots,A^n_{|\cV^n|}$ and it is reminiscent of the approximated MFG problem induced by \eqref{eq:fMn}. Each player aims at maximizing the objective functional
\begin{equation}\label{eq:approxCost}
	J^{i,n}(\gamma)=\E\bigg[\int_0^{\tauey^i} \int_A f(t,Y^i_t,\langle h,\rho^{N,n}_t\rangle,a)g^i _s(da)dt+G(\tauey^i,X^i_{\tauey^i})\bigg].
\end{equation}

 The aim of this section is to show that the approximated MFG problem provides $\varepsilon$-Nash equilibria for the $N$-player game \eqref{eq:approxSDE}-\eqref{eq:approxCost}. 
Shortly denote $\cC:=\cC([0,T];\R^{d})$.
Define $\psi:\R^d\times \cC^2\to \cC^3$ as 
\[\psi(x_0,w,w^0)\mapsto (x_0+\sigma w_t +\sigma^0 w^0_t\ ,\ w_t,\ w^0_t)_{t\in[0,T]}.
\] 	
\begin{notation}
	In this section, the canonical process on $\cC^3$ is denoted by $\varphi=(\varphi_t)_{t\in[0,T]}$ with components $\varphi:=(x,w,w^0)$. $\W$ denotes the Wiener measure on $\cC$.
\end{notation}

\begin{lemma}\label{lem:cont_tau1} The map $\Psi:\prob(\R^d\times\cC)\to\prob(\cC^3)$ which associate to each $Q$ the distribution of $\psi$ under $Q\otimes\W$, namely,
		\[\Psi(Q)( \cdot  ):=\E^{Q\otimes\W}\big[1_{\psi^{-1}(\cdot)}\big],\] 
is continuous when both spaces are endowed with their respective $\tau$-topology. 
\end{lemma}
\begin{proof}Let $\{Q_\alpha\} \subset \prob (\R^d \times \cC)$ be a net converging to $Q$ and $f:\cC^3\to\R$ bounded and measurable. By definition,
	\[
	\int_{\cC^3} f(\varphi) \Psi(Q_\alpha)(d\varphi)=\int_{\R^d\times \cC} \int_{\cC} f(\psi(x_0,w,w^0)) Q_\alpha(dx_0,dw)\W(dw^0).
	\]
	Using that $Q_\alpha\otimes\W$ converges to $Q\otimes\W$, the r.h.s.\ converges to
	\[
	\int_{\R^d\times \cC} \int_{\cC} f(\psi(x_0,w,w^0)) Q(dx_0,dw)\W(dw^0)=\int_{\cC^3} f(\varphi) \Psi(Q)(d\varphi),
	\]
	where for the last equality we used again the definition of $\Psi$.
\end{proof}
\begin{remark}\label{rmk:approx1}Let $\Psi$ be as in Lemma \ref{lem:cont_tau1}.
	\begin{enumerate} 
		\item For any input $Q$, the process $w^0$ is a Brownian motion under $\Psi(Q)$.
		\item Let $(X_0^i,W^i)$ be a sequence of i.i.d.\ initial conditions and Brownian motions on some probability space $(\Omega,\cF,\P)$. It is easy to see that
		\[
		\int f(\varphi) \Psi \bigg(\frac{1}{N}\sum_{i=1}^N\delta_{(X_0^i(\omega),W^i(\omega))}\bigg) (d\varphi) = \int_{\cC} \frac{1}{N}\sum_{i=1}^N f(X_0^i(\omega)+\sigma W^i(\omega)+\sigma^0 w^0, W^i(\omega),w^0)  \W (dw^0),
		\]
		for all measurable bounded functions $f: \cC^3 \to \mathbb \R$. 
	\end{enumerate}
\end{remark}

Fix $n\in\N$ and recall the fixed point $(\hat{\fm}^n_1,\ldots,\hat{\fm}^n_{|\cV^n|})$ from Corollary \ref{cor:fixedpointPhi}. This induces a flow of sub-probabilities $\mu^n$ as in \eqref{eq:induced_flow} through the random environment \eqref{eq:fMn}, which takes only finitely many values $\{\mu^{n,k}\}_{k=1,\ldots,|\cV^n|}$. 
From Remark \ref{rmk:approx1}, $w^0$ is a Brownian motion under any $\Psi (Q)$, therefore, the change of measure 
\begin{equation}\label{eq:def_psik}
\dfrac{d\Psi_k(Q)}{d\Psi (Q)}:=\sexp(U^{n,k})_T,\quad U^{n,k}_t(x,w,w^0):=\int_0^t\int_A (\sigma^0)^{-1} \tilde{b}(s,x_s,\mu^{n,k},a)\gamma^n_{s,x,w}(da)dw^0_s
\end{equation}
is well defined for every $k\in\{1,\ldots,|\cV^n|\}$. Moreover, using \ref{hp:cont_bdd} and Girsanov Theorem, the process $\tilde{w}^0_t:=w^0_t-\langle U^{n,k}\rangle_t$ is a Brownian motion under every $\Psi_k(Q)$. We denote by $\{\tilde{A}^n_1,\ldots,\tilde{A}^n_{|\cV^n|}\}$ the finite sigma-algebra generated by the discretization of $\tilde{w}^0$ as in Section \ref{sec:tight} and we set  $p_k:=\W(\tilde{A}^n_k)$.

\begin{lemma}\label{lem:cont_tau2} Let $k\in\{1,\ldots,|\cV^n|\}$.
	The map $\Psi_k:\prob(\R^d\times\cC)\to\prob(\cC^3)$ defined by $Q\mapsto\Psi_k(Q)$ as in \eqref{eq:def_psik} and the map $\hat{\Psi}_k:\prob(\R^d\times\cC)\to\prob(\cC^2)$ defined by
	\[
	\hat{\Psi}_k(Q)(B)=\Psi_k(Q)\big((x,w)^{-1}(B)\mid\tilde{A}^n_k\big), \qquad B\in\cB_{\cC^2}.
	\]
	are continuous when each space is endowed with the $\tau$-topology.
\end{lemma}

\begin{proof}
	As $n$ and $k$ are fixed we shortly denote $Z_T:=\sexp(U^{n,k})_T$ from \eqref{eq:def_psik} and observe that $Z_T$ does not depend on $Q$.
	Let $\{Q_\alpha\} \subset \prob (\R^d \times \cC)$ be a net converging to $Q$ and $f:\cC^3\to\R$ bounded and measurable. For any $K>0$,
	\begin{eqnarray*}
		\left|\int_{\cC^3} f (d\Psi_k(Q_\alpha)- d\Psi_k(Q))\right|&=&\left| \int_{\cC^3} Z_T\ f \ (d\Psi(Q_\alpha)-d\Psi(Q))\right|\\
		&\le&\left|\int_{\cC^3} Z_T\wedge K\ f\  (d\Psi(Q_\alpha)-d\Psi(Q))\right|\\
		&&+\left|\int_{\cC^3} (Z_T-Z_T\wedge K)\ f\ d\Psi(Q_\alpha)\right|\\
		&&+\left|\int_{\cC^3} (Z_T-Z_T\wedge K)\ f\ d\Psi(Q)\right|.
	\end{eqnarray*}
	The last two terms are bounded by a multiple of $\sup_{P\in\prob(\R^d\times\cC)}\int Z_T1_{\{Z_T>K\}}d\Psi(P)$ which converges to zero as $K\to\infty$ since
	\[
	\int (Z_T)^2d\Psi(P)=\E^P\left[e^{2U^{n,k}_T-\langle U^{n,k}\rangle_T}\right]\le e^{c^2T},
	\]
	where $c$ is an upper bound for $|\sigma^{-1}\tilde{b}|$ that exists from \ref{hp:cont_bdd}. For any $K$ fixed, the first term converges to zero from the continuity of $\Psi$ by Lemma \ref{lem:cont_tau1}.
	Thus, by taking first the limit $Q_{\alpha}\to Q$ and then $K\to\infty$ we obtain the desired convergence.
	
	For the second claim simply observe that for a bounded and measurable $f:\cC^2\to\R$,
		\[
	\E^{\hat{\Psi}_k(Q)}[f(x,w)]=\frac{1}{p_k}\E^{\Psi_k(Q)}\big[f(x,w)1_{A^n_k}(\tilde{w}^0)\big].
	\]
	 The continuity of $\Psi_k$ allows to conclude.
\end{proof}
\begin{remark}\label{rmk:image_through_Psi}
	\begin{enumerate}
		\item Let $(X_0^i,W^i)$ be a sequence of i.i.d.\ initial conditions and Brownian motions \ on some probability space $(\Omega,\cF,\P)$. Then,
		\[
		\hat{\Psi}_k\bigg(\frac{1}{N}\sum_{i=1}^N\delta_{(X_0^i(\omega),W^i(\omega))}\bigg)=	
		\frac{1}{N}\frac{1}{p_k}\sum_{i=1}^N\int_{A^n_k}\delta_{(X^i(\omega,\omega^0),W^i(\omega))}\W(d \omega^0),
		\]
		where $(X^1,\ldots,X^N)$ satisfies the following system of SDEs on $\Omega\times \cC$: 
		\begin{equation}\label{eq:SDEeqmeas}
			X^i_t=X_0^i+\int_0^t\int_A \tilde{b}(s,X^i_s,\mu^{n,k},a)\gamma^n_{s,X^i,W^i}(da)ds+\sigma W^i_t+\sigma^0 W^0_t
		\end{equation}
	and $W^0$ is a Brownian motion on $\cC$ independent of every $(X_0^i,W^i)$ for $i=1,\ldots,N$.
		\item Choosing $Q=\law(\xi)\otimes\W$, the distribution $\Psi(Q)$ in Lemma \ref{lem:cont_tau1} is that of a process satisfying $X_t=\xi+\sigma W_t+\sigma^0W^0_t$ with $(\xi,W,W^0)$ independent initial condition and Brownian motions. Given an input flow of sub-probabilities, the optimally controlled dynamics from  Proposition \ref{prop:exFBSDE} is obtained as a Girsanov transformation of $\Psi(Q)$. If we choose as input $\mu^n$, namely the one induced by the fixed point $(\hat{\fm}^n_1,\ldots,\hat{\fm}^n_{|\cV^n|})$, by definition of fixed point we have that the law of $(x,w)$ conditional to $A^n_k$ is again $\hat{\fm}^n_k$ (see also \eqref{eq:mnkfixpoint}). Since the change of measure in \eqref{eq:def_psik} is defined from the fixed point and its optimal control $\gamma^n$,  we have  
		\begin{equation}\label{eq:image_through_Psi_fix}
			\hat{\Psi}_k(\law(\xi)\otimes\W)=\hat{\fm}^n_k.
		\end{equation}
		
	\end{enumerate}
\end{remark}

We next consider the controlled dynamics $(Y^1,\ldots,Y^N)$ solution to \eqref{eq:approxSDE} with control $\gamma^n_{s,Y^i,W^i}(da)$ on a certain filtered probability space $(\Omega^{N,n},\cF^{N,n},\F^{N,n},\P^{N,n})$ and we recall that $\rho^{N,n}$, defined in \eqref{eq:rhoN}, takes only finitely many values as a function of $\omega^0$.
\begin{proposition}\label{prop:Sanov}
	Let $O$ be an open set, with respect to the $\tau$-topology on $\prob(\cC^2)$, containing $\hat{\fm}^n_k$. Then, 
	\[\lim_{N\to\infty}\P^{N,n}\bigg(\hat{\Psi}_k\bigg(\frac{1}{N}\sum_{i=1}^N\delta_{(X_0^i,W^i)}\bigg)\notin O\bigg)=0.
	\]
\end{proposition} 
\begin{proof}
	Consider first a probability space $(\Omega,\cF,\Q)$ supporting i.i.d sequences $(X^i_0,W^i)$ with $X^i_0\sim\xi$ and $W^i$ Brownian motions\ independent of $X^i_0$. Define $\zeta^n_k:=\hat{\Psi}_k(\frac{1}{N}\sum_{i=1}^N\delta_{(X_0^i,W^i)})$.
	We first show that, for any open set $O$ containing $\hat{\fm}^n_k$, \[
	\lim_{N\to\infty}\Q(\zeta^n_k\notin O)=0.
	\]
	From Sanov's Theorem (cf. \cite[Theorem 6.2.10]{dembo2010}) we know that for any open set $O'$,
	\[\limsup_{N\to\infty} \frac{1}{N}\log \Q\bigg(\frac{1}{N}\sum_{i=1}^N\delta_{(X_0^i,W^i)}\notin O'\bigg)\le \inf_{P\notin O'}\cH(P|\law(\xi)\otimes\W),
	\]
	where $\cH$ is the relative entropy. For any open set $O$,  apply the above result to $O'=(\hat{\Psi}_k)^{-1}(O)$, which is again open from Lemma \ref{lem:cont_tau2}.
	By observing that the sub-level sets $\{P : \cH(P\mid Q)\le a\}$ are compact for any $a\in\R$ and $Q\in\prob(\R^d\times\cC)$, see \cite[Lemma 2.1]{lacker2018}, we deduce that the infimum in the last inequality is attained by some measure in $O'$. Observe now that \eqref{eq:image_through_Psi_fix} in Remark \ref{rmk:image_through_Psi} ensures that $\law(\xi)\otimes\W\in\hat{\Psi}_k^{-1}(\hat{\fm}^n_k)$ and, by definition of relative entropy, $\cH(P|\law(\xi)\otimes\W)=0$ if and only if $P=\law(\xi)\otimes\W$. We conclude that $\inf_{P\notin (\hat{\Psi}_k)^{-1}(O)}\cH(P|\law(\xi)\otimes\W)>0$ if and only if $O$ contains $\hat{\fm}^n_k$, from which the claim follows. 
	
	The rest of the proof follows closely \cite[Theorem 2.6]{lacker2018}, we provide a sketch for completeness. Recall that $n\in\N$ is fixed. From Remark \ref{rmk:image_through_Psi} the random probability measure $\zeta^n_k$ is the empirical measure of \eqref{eq:SDEeqmeas} (conditional to $A^n_k$), where the drift depends on the flow of sub-probabilities  $\mu^n$ induced by the equilibrium measure $\P^n$ (see Notation \ref{notation:Pmu}). If we consider the change of measure $d \P^{N,n}/d\Q=\sexp(\int_0^\cdot \sum_{i=1}^N\Delta \tilde{b}^i_sdW^i_s)_T$ with
	\[\Delta\tilde{b}^i_t:=\int_A \sigma^{-1} \big(\tilde{b}(t,x_t,\rho^{N,n},a)-\tilde{b}(t,x_t,\mu^{n},a)\big)\gamma^n_{t,X,W}(da),
	\]
	we have that $\P^{N,n}\circ(X^1,\ldots,X^N)$ is a weak solution of \eqref{eq:approxSDE}.
	Fix $p,q\in(1,\infty)$ and denote by $p^*,q^*$ their conjugates. Assume $p$ and $q$ are such that $M = LTpq/2$ is an integer and define $t_j = jT/M$ for $j=0\ldots,M$. We show by iteration that for every $O$,
	\[\limsup_{N\to\infty} \frac{1}{N}\log \E^\Q\left[\P^{N,n}(\zeta^n_k\notin O'\mid \cF_{t_j})\right]\le-(p^*q^*)^{-(M-j)} \inf_{P\notin (\hat{\Psi}_k)^{-1}(O)}\cH(P|\law(\xi)\otimes\W).
	\]
	For $j=M$ the conditioning does not alter $\zeta^n_k$, since it is $\cF_T$-measurable and the result follows from the first part of the proof.
	Given the result for some $j$, the tower property of the conditional expectation and H\"{o}lder inequality yield
	\begin{eqnarray*}\limsup_{N\to\infty} \frac{1}{N}\log \E^\Q\left[\P^{N,n}(\zeta^n_k\notin O'\mid \cF_{t_{j-1}})\right]&\le&\limsup_{N\to\infty} \frac{1}{p^*q^*N}\log \E^\Q\left[\P^{N,n}(\zeta^n_k\notin O'\mid \cF_{t_j})\right]\\
		&&+\limsup_{N\to\infty} \frac{1}{p^*q^*N}\log \E^\Q\left[e^{\frac{pq}{2}\int_{t_{j-1}}^{t_j}\sum_{i=1}^N|\Delta\tilde{b}^i_t|^2dt}\right].
	\end{eqnarray*}
	Since $\Delta\tilde{b}^i$ are uniformly bounded, the second term is non positive and the iterative inequality follows.
\end{proof}

The previous convergence results allow us to obtain the main result of this section, namely, the existence of $\varepsilon$-Nash equilibria for the approximate problem \eqref{eq:approxSDE}-\eqref{eq:approxCost}. Recall that $\gamma^n$ is the relaxed control of Proposition \ref{prop:feedback}.
\begin{theorem}\label{thm:eps_Nash_approx}
	Let $n \geq 1$ be fixed. For every $\varepsilon>0$ there exists $N_{\varepsilon}\in\N$, such that, for any $N\ge N_{\varepsilon}$, the relaxed control $\gamma^{N,n}:=(\gamma^n_{\cdot,Y^1,W^1},\ldots,\gamma^n_{\cdot,Y^N,W^N})$ is an $\varepsilon$-Nash equilibrium for the $N$-player game \eqref{eq:approxSDE}-\eqref{eq:approxCost}.
\end{theorem}
\begin{proof}

	In this proof we need to consider two different vectors of relaxed controls, $\gamma^{N,n}$ given by the statement and $\gamma^{N,n,g}$ where only player $N$ deviates from $\gamma^{N,n}$ by choosing an alternative relaxed control $g_t = g(t,(Y^i)_{i=1}^N,(W^i)_{i=1}^N)$ instead of $\gamma^{n}_{t,Y^N,W^N}$. The choice of the $N$-th player is arbitrary and convenient for the notation. We start again on a probability space supporting i.i.d sequences $(X^i_0,W^i)$ with $X^i_0\sim\xi$, $W^i$ Brownian motions independent of $X^i_0$ and further independent of a Brownian motion $W^0$. We obtain the dynamics \eqref{eq:approxSDE} corresponding to $\gamma^{N,n}$ and $\gamma^{N,n,g}$ by the usual Girsanov transformations and we denote by $\P^{N,n}$ and $\P^{N,n,g}$ the corresponding probabilities. In both cases, the dynamics of the first $N-1$ players is induced by $g^i=\gamma^n_{\cdot,Y^i,W^i}$ for $i=1,\ldots,N-1$ and only the last equation is different. To distinguish the empirical sub-probabilities \eqref{eq:rhoN} in the two cases, we denote by $\rho^{N,n}$ the one corresponding to the vector of controls $\gamma^{N,n}$ and by $\rho^{N,n,g}$ the one where player $N$ deviates to $g$, i.e., where
	\[
	Y^N_t=Y^N_0+\int_0^t\int_A \tilde{b}(s,Y^N_s,\rho^{N,n,g},a) g_s(da)ds+\sigma W^N_t+\sigma^0 W^0_t.
	\]
	Our first aim is to show that $\langle u,\rho^{N,n}_s\rangle$ and $\langle u,\rho^{N,n,g}_s\rangle$ converge in probability to $\langle u,\mu^n_s\rangle$ as $N\to\infty$, for all $u$ measurable and bounded. Recall that $\mu^n$ is the equilibrium measure induced by the fixed point $(\hat{\fm}^n_1,\ldots,\hat{\fm}^n_{|\cV^n|})$, i.e. it is obtained via \eqref{eq:induced_flow} with $\fM^n$ as in \eqref{eq:fMn}. From Proposition \ref{prop:Sanov} we know that $\lim_{N\to\infty}\P^{N,n}(\zeta^n_k\notin O)=0$, where $\zeta^n_k=\hat{\Psi}_k(\frac{1}{N}\sum_{i=1}^N\delta_{(X_0^i,W^i)})$ and $O$ is an open set containing $\hat{\fm}^n_k$. Note that under $\P^{N,n}$, 
	\[\zeta^n_k= \frac{1}{N}\frac{1}{p_k}\sum_{i=1}^N\int_{A^n_k}\delta_{(Y^i(\omega,\omega^0),W^i(\omega))}\W(d \omega^0),
	\]
	where $(Y^1,\ldots,Y^N)$ satisfies \eqref{eq:approxSDE} with vector of controls $\gamma^{N,n}$. Due to \eqref{eq:rhoN}, on every $A^n_k$, we have that $\langle u, \rho^{N,n}_s\rangle=\int_{\cC([0,T];\R^{d})} u(x_s)1_{\{\tau(x)>s\}}\zeta^n_k(dx)$. The convergence in $\tau$-topology to $\hat{\fm}^n_k$ for every $A^n_k$ and the fact that the partition $A^n_1,\ldots,A^n_{|\cV^n|}$ is finite imply that $\langle u,\rho^{N,n}_s\rangle$ converges to $\langle u,\mu^n_s\rangle$ as $N\to\infty$. The same conclusion holds for $\langle u,\rho^{N,n,g}_s\rangle$. Indeed, recalling that $b$ is bounded from Assumption \ref{ass:alltogether}, we deduce
	\[\P^{N,n,g}(\zeta^n_k\notin O)=\E^{N,n}\bigg[\frac{d\P^{N,n,g}}{d\P^{N,n}}1_{\{\zeta^n_k\notin O\}}\bigg]\le C \P^{N,n}(\zeta^n_k\notin O),
	\] for some $C>0$. Using Proposition \ref{prop:Sanov} we have $\lim_{N\to\infty}\P^{N,n,g}(\zeta^n_k\notin O)=0$ and the same proof applies.
	
	We use this fact to show that $Q^{N,n,g}:=\P^{N,n,g}\circ(Y^N)^{-1}$ converges to the law of the solution to
	\begin{equation}\label{eq:limiting_approx}
		X_t=X_0+\int_0^t\int_A \tilde{b}(s,X_s,\mu^n,a) g_s(da) ds+\sigma W_t+\sigma^0 W^0_t,
	\end{equation}
	that we denote by $Q^{n}$. More precisely, by calculating the relative entropy we obtain 
	\[
	\cH(Q^{N,n,g}\mid Q^{n})=\frac{1}{2}\E^{Q^{N,n,g}}\left[\int_0^{\tauex}|\sigma^{-1}\Delta^n \tilde{b}_t|^2 dt\right],
	\]
	where $\Delta^n \tilde{b}_t:=\int_A (\tilde{b}(s,Y^N_s,\mu^n,a)-\tilde{b}(s,Y^N_s,\rho^{N,n,g},a)) g_s(da) ds$. Recall now that $b$ is continuous and bounded by Assumption \ref{ass:alltogether} and $\tilde{b}$ defined in \eqref{eq:btilde} is given by
	\[
	\tilde{b}(t,x,m,a):=b(t,x,\langle h,m_t\rangle,a)+k(0)\langle 1,m_t\rangle+\int_0^tk'(t-s)\langle 1, m_s\rangle ds.
	\]
	Using the convergence in probability of $\langle u,\rho^{N,n,g}_s\rangle$ to $\langle u,\mu^n_s\rangle$ for all $u$ measurable and bounded, together with the continuity and boundedness of $b$, we obtain 	$\cH(Q^{N,n,g}\mid Q^{n})\to 0$, as $N\to\infty$.
	
	We can then proceed as in \cite{CGL19} for the construction of $\varepsilon$-Nash equilibria for the approximated problem. Denote by $\cG$ the class of relaxed control and consider, as above, the vectors $\gamma^{N,n}$ and $\gamma^{N,n,g}$ where only player $N$ chooses a different $g\in\cG$. Recall also that $J^{N,n}$ is the objective functional of player $N$ in the approximated game as defined in \eqref{eq:approxCost} and $J^{n}$ is the mean field one as in \eqref{eq:MFG_cost_relax_n}. To conclude the proof we need to show:
	\begin{enumerate}
		\item $\lim_{N\to\infty} J^{N,n}(\gamma^{N,n})=J^n(\gamma^n)$;
		\item For any $g\in\cG$ such that $J^{N,n}(\gamma^{N,n,g})\ge \sup_{\tilde{g}\in\cG}J^{N,n}(\gamma^{N,n,\tilde{g}})-\varepsilon/2$, we have \[\limsup_{N\to\infty}J^{N,n}(\gamma^{N,n,g})\le J^n(\gamma^n)\ ;\]
		\item $J^{N,n}(\gamma^{N,n})\ge  \sup_{\tilde{g}\in\cG}J^{N,n}(\gamma^{N,n,\tilde{g}})-\varepsilon$, for $N$ large enough.
	\end{enumerate}
	The first item is the convergence of the value function when using the control $\gamma^n$ in the $N$-player game and in the mean field limit. The convergence of the state dynamics to the mean field limit \eqref{eq:limiting_approx} has been shown in the first part of the proof for an arbitrary control $g\in\cG$. Similarly, since all the functions defining $J^{N,n}$ and $J^n$ are bounded and continuous by Assumption \ref{ass:alltogether}, the convergence of the value functions is again a direct consequence of the convergence of $\langle u,\rho^{N,n}_s\rangle$ to $\langle u,\mu^n_s\rangle$ for every bounded and measurable $u$.
	
	For the second item, we use again the first part of the proof to ensure the convergence of $J^{N,n}(\gamma^{N,n,g})$ to $J^n(g)$. Using the fact that $\gamma^n$ is optimal for the mean field problem, we deduce
	\[
	\limsup_{N\to\infty}J^{N,n}(\gamma^{N,n,g})=J^n(g)\le \sup_{\tilde{g}\in\cG}J^n(\tilde{g})=J^n(\gamma^n).
	\]
	
	For the third item, 
	\[
	J^{N,n}(\gamma^{N,n})- \sup_{\tilde{g}\in\cG}J^{N,n}(\gamma^{N,n,\tilde{g}})\ge J^{N,n}(\gamma^{N,n})-J^n(\gamma^n)+J^n(\gamma^n)-J^{N,n}(\gamma^{N,n,g})-\varepsilon/2.
	\]
	From the first two items we can choose $N$ large enough so that the conclusion follows.
\end{proof}

\subsection{Uniform approximation}
In this subsection, we work again in the framework introduced in Section \ref{sec:MFG} under the following:
\begin{assumption} In addition to Assumption \ref{ass:alltogether}, we assume that
	\label{ass:Nasheq}
	\begin{enumerate}[label=\textnormal{\textbf{(H\arabic*)}},ref=(H\arabic*)] \setcounter{enumi}{4}
		\item\label{ass:lipb} $b(t,x,m,a)$ is Lipschitz continuous in the real variable $m$.
		\item\label{ass:f_indep_meas} The running cost is of the form $f(t,x)$.
	\end{enumerate}
\end{assumption}
All the results of the previous section are still valid under \ref{ass:f_indep_meas} thanks to Remark \ref{rmk:set-valued} and, in particular, they can be satisfied by the model of Section \ref{sec:motivation}.

For any $N,n\in\N$ and for some relaxed control $\gamma$, denote by $\P^{N,\gamma}$ the law of the process $(X^1,\ldots,X^N)$ satisfying $X^i_0=\xi^i$ and
	\begin{equation}\label{eq:SDErelax}
	X^i_t=X^i_0+\int_0^t\int_A \tilde{b}(s,X^i_s,\mu^N,a)\gamma_s (da)ds+\sigma W^i_t+\sigma^0 W^0_t,
\end{equation}
which is the relaxed form of \eqref{eq:model2}.
Respectively, $\P^{N,n,\gamma}$ is the law of the process $(Y^{1,n},\ldots,Y^{N,n})$ satisfying \eqref{eq:approxSDE}. As commented earlier, the difference is that in \eqref{eq:approxSDE} the players interact with the empirical sub-distribution of players who have not been absorbed integrated with respect to the common noise component on a finite number of events $A^n_1,\ldots,A^n_{|\cV^n|}$. Recall that such sets are induced by the random variable $V^n$ constructed at the beginning of Section \ref{sec:tight} which is a discretization of the common noise for which \eqref{eq:discret_points} is true.
\begin{lemma}\label{lem:unif_entropy}
	$\cH(\P^{N,\gamma}|\P^{N,n,\gamma})\to 0$ as $n\to\infty$, uniformly in $N$ and in $\gamma$.
\end{lemma}
\begin{proof}
	Fix arbitrary $N$ and $\gamma$.
	Let $\P$ be the probability under which the system of SDEs of \ref{eq:model2} satisfies $X^i_t=\xi^i+\sigma W^i_t+\sigma^0W^0_t$, for each $i=1,\dots,N$ (obtained by a Girsanov transformation of $\P^{N,\gamma}$).
	Define 
	\begin{eqnarray*}
		\Delta^n \tilde{b}_t&:=&\int_A\left(b(t,X_t,\langle h,\rho^{N,n}_t\rangle,a)-b(t,X_t,\langle h,\mu^N_t\rangle,a)\right)\gamma_t (da)\\
		&&+k(0)\langle 1,\rho^{N,n}_t-\mu^N_t\rangle+\int_0^tk'(t-s)\langle 1, \rho^{N,n}_s-\mu^N_s\rangle ds
	\end{eqnarray*}
	and observe that 
	\begin{equation}\label{eq:unifconv}
	\cH(\P^{N,\gamma}\mid\P^{N,n,\gamma})=\frac{1}{2}\E^{N,\gamma}\left[\int_0^{\tauex}|\sigma^{-1}\Delta^n \tilde{b}_t|^2 dt\right]=\frac{1}{2}\E\left[\frac{d\P^{N,\gamma}}{d\P}\int_0^{\tauex}|\sigma^{-1}\Delta^n \tilde{b}_t|^2 dt\right].
	\end{equation}
	 Using H\"{o}lder inequality, \ref{ass:lipb} and the smoothness of $k$ from Definition \ref{def:kernel}, we obtain
	 \begin{align*}
	 	\E\left[\frac{d\P^{N,\gamma}}{d\P}\int_0^{\tauex}|\sigma^{-1}\Delta^n \tilde{b}_t|^2 dt\right]^2\le C\E\left[\left(\frac{d\P^{N,\gamma}}{d\P}\right)^2\right]\E\Bigg[&\bigg(\int_0^{\tauex} \big|\langle h,\rho^{N,n}_t\rangle-\langle h,\mu^N_t\rangle\big|^2\\&
	 +\langle 1,\rho^{N,n}_t-\mu^N_t\rangle^2\\
	 &+\bigg(\int_0^t\langle 1, \rho^{N,n}_s-\mu^N_s\rangle^2 ds\bigg)dt\bigg)^2\Bigg]
	 \end{align*}
 for some constant $C>0$. The first expectation in the last product is uniformly bounded from Assumption \ref{ass:alltogether}. We show that the second expectation converges to zero uniformly in $N$ and $\gamma$. Towards this goal we focus on the pointwise convergence to zero of the term $|\langle h,\rho^{N,n}_t\rangle-\langle h,\mu^N_t\rangle|$, as the convergence of the other terms is completely analogous. The dominated convergence theorem yields the desired result.

Recalling that $\tau(x)=\inf\{t\in[0,T]\ :\ x_t\notin \cO\}$, we have
	\[|\langle h,\rho^{N,n}_t\rangle-\langle h,\mu^N_t\rangle|\le\frac{1}{N}\sum_{i=1}^N\left|h(X^i_t)1_{\{\tau(X^i)>t\}}-\sum_{k=1}^{|\cV^n|}\E\big[ h(X^i_t)1_{\{\tau(X^i)>t\}}\mid \cF^{\xi,W}_t\vee \sigma(A^n_k)\big]1_{A^n_k}\right|,
	\]
	where $\cF^{\xi,W}_t=\sigma(\xi^i,\ W^i_s, \, 0\le s\le t, \,  i=1,\ldots, N)$.
	Let $\tilde{Y}^i_t:=\xi^i+\sigma W^i_t+\sigma^0 V^n_t$. We add and subtract $h(\tilde{Y}^i)1_{\{\tau(\tilde{Y}^i)>t\}}$ in the r.h.s.\ and use the measurability of $\tilde{Y}^i$ to obtain
	\begin{eqnarray*}
		|\langle h,\rho^{N,n}_t\rangle-\langle h,\mu^N_t\rangle|&\le&\frac{1}{N}\sum_{i=1}^N\Bigg(\sum_{k=1}^{|\cV^n|}\E\bigg[\big|h(\tilde{Y}^i_t)1_{\{\tau(\tilde{Y}^i)>t\}}-h(X^i_t)1_{\{\tau(X^i)>t\}}\big|\mid \cF^{\xi,W}_t\vee \sigma(A^n_k)\bigg]1_{A^n_k}\\
		&&+\big|h(X^i_t)1_{\{\tau(X^i)>t\}}-h(\tilde{Y}^i_t)1_{\{\tau(\tilde{Y}^i)>t\}}\big|\Bigg).
	\end{eqnarray*}
 We now rewrite
 \[h(X^i_t)1_{\{\tau(X^i)>t\}}-h(\tilde{Y}^i_t)1_{\{\tau(\tilde{Y}^i)>t\}}=\big(h(X^i_t)-h(\tilde{Y}^i_t)\big)1_{\{\tau(X^i)>t\}}+h(\tilde{Y}^i_t)\big(1_{\{\tau(X^i)>t\}}-1_{\{\tau(\tilde{Y}^i)>t\}}\big).\]
	Using the Lipschitz continuity of $h$ from \ref{hp:cont_bdd},  
	\[\big|h(X^i_t)-h(\tilde{Y}^i_t)\big|1_{\{\tau(X^i)>t\}}\le L|\sigma^0(V^n_t-W^0_t)|
	\] for some constant $L$. The last term converges pointwise to zero as $n\to\infty$, uniformly in $N$ and $\gamma$. As for the term
	\[h(\tilde{Y}^i_t)\big(1_{\{\tau(X^i)>t\}}-1_{\{\tau(\tilde{Y}^i)>t\}}\big)\]
 using again the pointwise convergence of $V^n$ and the $\P$-a.s.\ continuity of the function $x\mapsto 1_{\{\tau(x)>t\}}$ (see \cite[Lemma C.3]{CF18}) we deduce that it also converges to zero as $n\to\infty$, uniformly in $N$ and $\gamma$. The exact same argument applies to any other Lipschitz function $h$, in particular to the constant function $1$.  
 
 To conclude, we use \eqref{eq:unifconv} to deduce that $\cH(\P^{N,\gamma}\mid\P^{N,n,\gamma})\to 0$ as $n\to\infty$, uniformly in $N$ and $\gamma$.
\end{proof}
\begin{corollary} \label{cor:unifApprox}
	$|J^{N}(\gamma)-J^{N,n}(\gamma)|\to 0$ as $n\to\infty$, uniformly in $N$ and $\gamma$.
\end{corollary}
\begin{proof}
	From Assumption \ref{ass:f_indep_meas},
	\begin{align*}
		|J^{N}(\gamma)-J^{N,n}(\gamma)|&=\bigg|\E^N\bigg[\int_0^{\tauex^N}f_1(t,X^N_t)dt\bigg]-\E^{N,n}\bigg[\int_0^{\tauex^N}f_1(t,X^N_t)dt\bigg]\bigg|\\
		&\quad +\bigg|\E^N\bigg[G(\tauex^N,X^N_{\tauex^N})\bigg]-\E^{N,n}\bigg[G(\tauex^N,X^N_{\tauex^N})\bigg]\bigg| .
	\end{align*}
	From Lemma \ref{lem:unif_entropy} and Pinsker's inequality $\P^{N,n,\gamma}$ converges to $\P^{N,\gamma}$ in total variation, uniformly in $N$ and in $\gamma$. This implies, in particular, uniform convergence in the $\tau$-topology. Since all the cost functions are bounded and measurable the result follows.
\end{proof}

We are now ready to prove the existence of $\varepsilon$-Nash equilibria for the problem \eqref{eq:model2}-\eqref{eq:model2_cost}.  Recall that, for any $n\in\N$, $\gamma^n$ is the relaxed control of Proposition \ref{prop:feedback}.
\begin{theorem}
	Under Assumption \ref{ass:Nasheq}, for every $\varepsilon>0$ there exists $n_{\varepsilon},N_{\varepsilon}\in\N$, such that, for any $N\ge N_{\varepsilon}$, the relaxed control $\gamma^{N,n_{\varepsilon}}:=(\gamma^{n_{\varepsilon}}_{\cdot,Y^1,W^1},\ldots,\gamma^{n_{\varepsilon}}_{\cdot,Y^N,W^N})$ is an $\varepsilon$-Nash equilibrium for the $N$-player game \eqref{eq:model2}-\eqref{eq:model2_cost}.
\end{theorem}
\begin{proof} For any $n\in\N$, and for $\gamma^{N,n}=(\gamma^{n}_{\cdot,Y^1,W^1},\ldots,\gamma^{n}_{\cdot,Y^N,W^N})$, let $\gamma^{N,n,g}$ be the vector of relaxed controls when player $N$ plays $g_t = g(t,(Y^i)_{i=1}^N,(W^i)_{i=1}^N)$ instead of $\gamma^{n}_{t,Y^N,W^N}$. Using the uniform approximation of Corollary \ref{cor:unifApprox}, we can find $n_{\varepsilon}\in\N$ such that  
\[
J^{N}(\gamma^{N,n_{\varepsilon}})- \sup_{\tilde{g}\in\cG}J^{N}(\gamma^{N,n_{\varepsilon},\tilde{g}})\ge J^{N,n_{\varepsilon}}(\gamma^{N,n_{\varepsilon}})- \sup_{\tilde{g}\in\cG}J^{N,n_{\varepsilon}}(\gamma^{N,n_{\varepsilon},\tilde{g}})-\varepsilon/2.
\]
From Theorem \ref{thm:eps_Nash_approx}, we can find $N_\varepsilon$ such that $J^{N,n_{\varepsilon}}(\gamma^{N,n_{\varepsilon}})- \sup_{\tilde{g}\in\cG}J^{N,n_{\varepsilon}}(\gamma^{N,n_{\varepsilon},\tilde{g}})\ge -\varepsilon/2$, for any $N\ge N_{\varepsilon}$, as desired.
\end{proof}

\appendix
\section{Results on BSDEs in general spaces}

We recall here some well known results on BSDE on general probability spaces and we adapt them to our framework.

\begin{theorem}\label{thm:BSDEapp} Let $(\Omega,\cF,\F=\{\cF_t\},\P)$ be a filtered probability space satisfying the usual assumptions, supporting a $k$-dimensional Brownian Motion $\{W_t\}_{t\ge 0}$ and such that $L^2(\Omega,\F,P)$ is separable.
	Let $F:\Omega\times\R\times \R^{k}\to\R$ be progressively measurable and such that there exists $L>0$ satisfying
	\[|F(\omega,t,z_1)-F(\omega,t,z_2)|\le L|z_1-z_2|,\quad  \forall\omega\in\Omega,\ \forall t\ge 0,\ \forall z_1,z_2\in\R^k.  
	\]
	Let $\tau$ be a stopping time bounded by $T>0$ and $Q$ be an $\cF_\tau$-measurable random variable. There exists a unique $(Y,Z,M)$ 
	\begin{equation}\label{eq:BSDE_app}
		Y_t = Q+\int_{t}^\tau F(s,Z_s) dt-\int_{t}^\tau Z_s dW_s-\int_{t}^\tau dM_s,
	\end{equation}
	where $Y$ is a c\`adl\`ag adapted process, $Z$ a predictable process with $\E[\int_0^T|Z_s|^2ds]<\infty$ and a c\`adl\`ag martingale orthogonal to $W$.
	The triple $(Y,Z,M)$ is called the solution to the BSDE \eqref{eq:BSDE_app} with coefficients $(F,Q)$.
	
	Moreover, for $i=1,2$, let $(F^i,Q^i)$ with the above properties and denote by $(Y^i,Z^i,M^i)$ the corresponding unique solutions. If, in addition,
	\begin{enumerate}
		\item $Q^1\ge Q^2$ $\P$-a.s.,\label{ass:app1}
		\item $F^1(\omega,t,Z^2_t(\omega))\ge F^2(\omega,t,Z^2_t(\omega))$  $dt\otimes\P$-a.s.,\label{ass:app2}
	\end{enumerate}
then $Y^1\ge Y^2$ up to $dt\otimes\P$ null sets.
\end{theorem}
\begin{proof}
	From \cite[Theorem 6.1]{bookElkaroui1997backward} we deduce existence, uniqueness and the square integrability of the process $Z$.
We give a proof of the comparison result as a straightforward adaptation of the argument of \cite[Thoerem 1]{CohenElliot_comparison}, which is proved for the case of a deterministic terminal time. Let $\tau$ be a given stopping time bounded by $T>0$. Take $(F^i,Q^i)$ and $(Y^i,Z^i,M^i)$ for $i=1,2$ as in the statement. By taking the difference of the equations of the form \eqref{eq:BSDE_app}, satisfied by the respective solutions, we have
\[
Y^1_t-Y^2_t +\int_{t}^\tau (F^2(s,Z^2_s)-F^1(s,Z^1_s)) dt+\int_{t}^\tau (Z^1_s-Z^2_s) dW_s+\int_{t}^\tau dM^1_s-\int_{t}^\tau dM^2_s= Q^1-Q^2,
\]
where the r.h.s.\ is non-negative by assumption \ref{ass:app1}.
By adding and subtracting $F^1(s,Z^2_s)$ to the l.h.s, the above inequality can be rewritten as
\[
Y^1_t-Y^2_t\ge\int_{t}^\tau \big(F^1(s,Z^2_s)-F^2(s,Z^2_s)\big) dt+X_\tau-X_t
\]
with
\[
	X_u:=\int_0^u \big(F^1(s,Z^1_s)-F^1(s,Z^2_s)\big) ds+\int_0^u (Z^2_s-Z^1_s) dWs+ \int_0^u dM^2_s-\int_0^u dM^1_s,\qquad u\in[0,T].
\]
We show below that $(X_t)_{t\in[0,T]}$ is a martingale under an equivalent probability $\tilde{\P}$. Given the claim, by taking conditional expectation with respect to $\cF_t$ in the above inequality and using assumption \ref{ass:app2}, the r.h.s.\ in non-negative. It follows $Y^1_t-Y^2_t\ge 0$ $\tilde{\P}$-a.s., hence, $\P$-a.s. by equivalence. Since the processes $Y^i$ are c\`adl\`ag, we conclude that $Y^1-Y^2$ is a non-negative process up to indistinguishability. 

The thesis follows if we prove the claim. Shortly denote $\Delta F_t:=F^1(t,Z^2_t)-F^1(t,Z^1_t)$ and $\Delta Z_t:=1_{\{Z^2_t-Z^1_t\neq 0\}}(Z^2_t-Z^1_t)/|Z^2_t-Z^1_t|^2$. We define the equivalent measure $\tilde{\P}$ by $\frac{d\tilde{\P}}{d\P}=U_T$, where 
\[
U_t:=\exp\left(\int_0^t\Delta F_s \Delta Z_s^T dW_s-\frac{1}{2}\int_0^t |\Delta F_s \Delta Z_s|^2ds\right).
\]
From the uniform Lipschitz continuity of $F^1$ the Novikov condition is satisfied and the process $(U_t)_{t\in[0,T]}$ is a martingale. By Girsanov's Theorem, $\widetilde{W}_t:=W_t-\int_0^t \Delta F_s \Delta Z_sds$ is a Brownian motion with respect to $\tilde{\P}$. We then observer that, under $\tilde{\P}$,
\begin{eqnarray*}
X_t&=&\int_t^\tau F^1(s,Z^1_s)-F^1(s,Z^2_s) ds+\int_t^\tau (Z^2_s-Z^1_s) dW_s+ \int_t^\tau dM^2_s-\int_t^\tau dM^1_s,\\
&=&\int_t^\tau (Z^2_s-Z^1_s) d\widetilde{W}_s+ \int_t^\tau dM^2_s-\int_t^\tau dM^1_s,
\end{eqnarray*}
which is a martingale under $\tilde{\P}$ since $M^1,M^2$ are orthogonal to $W$ and $Z^2-Z^1$ is a predictable square integrable process. To see the latter recall that the quadratic variation process is invariant under an equivalent probability measure (\cite[Theorem III.3.13]{bookJacod}), therefore, the one of $\int_0^t (Z^2_s-Z^1_s) d\widetilde{W}_s$ coincide with $\E[\int_0^t|Z^2_s-Z^1_s|^2ds]<\infty$ for any $0\le t\le T$. We deduce that the integrand is also square-integrable with respect to $\tilde{\P}$ (\cite[Proposition I.4.50 and Theorem III.4.5]{bookJacod}).
\end{proof}
\bibliographystyle{abbrv}
\bibliography{MFG_bank_run}
\end{document}